\documentclass[a4paper, 11pt]{amsart}
\usepackage{amsmath, amssymb, amsfonts, amsthm, enumerate, mathtools, tikz-cd, hyperref}
\usepackage{hyperref}
\usepackage[margin=1.1in]{geometry}
\usepackage[backend=biber]{biblatex}
\renewbibmacro{in:}{}
\usepackage{booktabs}
\usepackage{bm}
\usepackage{bbm}
\usepackage{dirtytalk}
\usepackage{tikz}
\usepackage{tikz-cd}
\usepackage{pgf}
\usepackage{graphicx}
\usepackage{wrapfig}
\usepackage{ragged2e}
\usepackage{mathabx}
\usepackage{mathrsfs}
\usepackage{multirow}
\usepackage{multicol}
\usepackage[T1]{fontenc}
\usepackage{stmaryrd}

\usepackage{color,soul}
\usepackage{mathtools}
\usepackage{amsmath}
\usepackage{amsthm}
\usepackage{tikz-cd}
\usepackage[T2A,T1]{fontenc}
\numberwithin{equation}{section}

\newtheorem{thm}{Theorem}[section]
\newtheorem{lem}[thm]{Lemma}
\newtheorem{defi}[thm]{Definition}
\newtheorem{rem}[thm]{Remark}

\newtheorem{prop}[thm]{Proposition}

\newtheorem{assum}[thm]{Assumption}

\newtheorem{thmx}{Theorem}
\newtheorem{lemx}{Lemma}

\newtheorem*{Not}{Notation}
\newtheorem*{thmt}{Theorem}

\newcommand{\NN}{\mathbb{N}}	
\newcommand{\ZZ}{\mathbb{Z}}	
\newcommand{\QQ}{\mathbb{Q}}	
\newcommand{\BB}{\mathbb{B}}
\newcommand{\CC}{\mathbb{C}}
\newcommand{\RR}{\mathbb{R}}
\newcommand{\AAA}{\mathbb{A}}

\newcommand{\HH}{\mathcal{H}}

\newcommand{\Dcris}{\mathbb{D}_{\mathrm{crys}}}

\newcommand{\GG}{\Gamma}

\newcommand{\OO}{\mathcal{O}}
\newcommand{\vp}{\varphi}

\newcommand{\Qp}{\QQ_p}

\newcommand{\Zp}{\ZZ_p}

\newcommand{\Brig}{\BB^{+}_{\mathrm{rig},\Qp}}
\newcommand{\BrigE}{\BB^{+}_{\mathrm{rig},E}}

\newcommand{\Mbar}{\underline{M}}

\newcommand{\MatA}{\begin{pmatrix}
0 & -1/vp^{k+1} \\[6pt]
1 & a/vp^{k+1}
\end{pmatrix}}
\newcommand{\MatQ}{\begin{pmatrix}
\alpha & -\beta \\[6pt]
-vp^{k+1} & vp^{k+1}
\end{pmatrix}}

\newcommand{\adeles}{\mathbb{A}_{F}}

\newcommand{\PP}{\mathfrak{p}}
\newcommand{\PPP}{\overline{\mathfrak{p}}}

\newcommand{\qq}{\mathfrak{q}}
\newcommand{\pset}{\{\mathfrak{p},\overline{\mathfrak{p}}\}}
\newcommand{\qqq}{\mathfrak{q}}

\newcommand{\FF}{\mathcal{F}}

\newcommand{\Of}{\mathcal{O}_{F}}
\newcommand{\gl}{\mathrm{GL}_{2}}
\newcommand{\YQ}{Y_{\mathbb{Q}}}
\newcommand{\YF}{Y_{F}}
\newcommand{\YFs}{Y^{*}_{F}}
\newcommand{\YQQ}{Y_{\mathbb{Q},1}}
\newcommand{\YFss}{Y^{*}_{F,1}}

\newcommand{\frkn}{\mathfrak{N}}
\newcommand{\aaa}{\mathfrak{a}}
\newcommand{\frkm}{\mathfrak{m}}
\newcommand{\Las}{L^{\mathrm{As}}}
\newcommand{\Laspo}{L^{\mathrm{As}}(\Psi,\theta,s)}
\newcommand{\frkl}{\mathfrak{l}}

\newcommand{\Ho}{\mathrm{H}^1}
\newcommand{\Ht}{\mathrm{H}^2}
\newcommand{\cgn}{\prescript{}{c}{g_N}}
\newcommand{\cC}{\prescript{}{c}{C}}
\newcommand{\hh}{\mathbb{H}}
\newcommand{\kam}{\kappa_{a/m}}
\newcommand{\asaieisen}{\prescript{}{c}{\Xi}}
\newcommand{\cPhi}{\prescript{}{c}{\Phi}}
\newcommand{\htt}{\mathrm{H}^{2}(\YFss(\frkn),\ZZ)}
\newcommand{\sym}{\mathrm{Sym}}
\newcommand{\tsym}{\mathrm{TSym}}
\newcommand{\mom}{\mathrm{mom}}
\newcommand{\CG}{\mathrm{CG}}
\newcommand{\Zpn}{(\ZZ/p^{n})^{\times}}
\newcommand{\logtu}{\frac{\log_{p}(t)}{\log_{p}(u)}}

\newcommand{\czeta}{\prescript{}{c}{\mathcal{Z}}}
\newcommand{\ceis}{\prescript{}{c}{\mathrm{Eis}}}

\newcommand{\matt}{\begin{pmatrix} 1 & -ta \\ 0 & 1 \end{pmatrix}}
\newcommand{\matjt}{\begin{pmatrix} t & 0 \\ 0 & 1 \end{pmatrix}}
\newcommand{\mata}{\begin{pmatrix} 1 & -a \\ 0 & 1 \end{pmatrix}}
\newcommand{\Zpr}{(\mathbb{Z}/p^{r})^\times}
\newcommand{\logt}{\log_{u}(t)}
\newcommand{\asaipaddel}{\prescript{}{c}{L^{\mathrm{As},\delta}_{p}}}
\newcommand{\asaipad}{\prescript{}{c}{L^{\mathrm{As}}_{p}}}
\newcommand{\rta}{\tilde{\alpha}}
\newcommand{\rtb}{\tilde{\beta}}

\DeclareMathOperator{\Gal}{Gal}

\usepackage{color,xcolor}

\begin{document}
\baselineskip 18pt

\title[Non-ordinary $p$-adic Asai $L$-functions of BMF]{On $p$-adic Asai $L$-functions of Bianchi modular forms at non-ordinary primes and their decomposition into bounded $p$-adic $L$-functions}
\author{Mihir V. Deo} 
\email{mdeo048@uottawa.ca}
\address{Department of Mathematics and Statistics, University of Ottawa, Ottawa, ON, Canada K1N 6N5}
\date{}

\keywords{Iwasawa theory, $p$-adic $L$-functions, $p$-adic Hodge theory, Bianchi modular forms, Asai $L$-functions}

\begin{abstract}
Let $p$ be an odd prime integer, $F/\QQ$ be an imaginary quadratic field, and $\Psi$ be a small slope cuspidal Bianchi modular form over $F$ which is non-ordinary at $p$. 
In this article, we first construct a $p$-adic distribution $L^{\mathrm{As}}_{p}(\Psi)$ that interpolates the twisted critical $L$-values of Asai (or twisted tensor) $L$-function of $\Psi$, generalizing the works of Loeffler--Williams from the ordinary case to the non-ordinary case.
To obtain this distribution, we construct some polynomials using Asai--Eisenstein elements: the Betti analogue of the Euler system machinery, developed by Loeffler--Williams. We use some techniques analogous to those of Loeffler--Zerbes for interpolating the twists of Beilinson--Flach elements arising in the Euler system associated with Rankin--Selberg convolutions of elliptic modular forms.
We also use the interpolation method developed by Amice--V\'elu, Perrin-Riou, and B\"uy\"ukboduk--Lei
in the construction. 
Furthermore, under some assumptions, we decompose these unbounded $p$-adic distributions into the linear combination of bounded measures as done by Pollack, Sprung, and Lei--Loeffler--Zerbes in the elliptic modular forms case.
\end{abstract}
\maketitle

\medskip

\section{Introduction}
Investigating complex $L$-functions is an important and active area of number theory.
More precisely, we are interested in the special values of $L$-functions, since several conjectures in number theory, including the Bloch--Kato conjecture, describe important arithmetic data in terms of the special (critical) values of $L$-functions.
One of the most important tools for understanding the critical values of complex $L$-functions is the construction and study of the $p$-adic avatars of these $L$-functions, i.e., $p$-adic $L$-functions.
A one-variable $p$-adic $L$-function is a measure or a distribution on $\Zp^{\times}$ that interpolates critical values of a given complex $L$-function.
Using the Amice transform, we can associate a $p$-adic measure with a power series in the Iwasawa algebra $E\otimes\OO_{E}[[T]]$ (or in $E \otimes \OO_{E}[[\Zp^{\times}]]$), where $E$ is a finite extension of $\Qp$ and $\OO_E$ is its ring of integers. 
Similarly, for $w\in\RR_{\geq 0}$, a \emph{$w$-tempered} distribution can be associated with a power series in the distribution space $\HH_{E,w}\subset E[[T]]$, with unbounded denominators and growth $O(\log^{w}_{p})$. 
See Section \ref{patchsec2} for details about $\HH_{E,w}$.

In \cite{Coates} and \cite{CPR}, Coates and Perrin-Riou formulated a conjecture that predicts the existence of $p$-adic $L$-functions associated to any motive over $\QQ$ having at least one critical $L$-value. 
Loeffler and Williams proved this conjecture for the (conjectural) Asai motive attached to a $p$-ordinary Bianchi eigenform in \cite{CW}. 
More precisely, for a $p$-ordinary cuspidal Bianchi modular form $\Psi$, they constructed a $p$-adic measure $\Las_{p}(\Psi)$ that interpolates the critical $L$-values of the Asai $L$-function attached to $\Psi$. 
Moreover, their construction is independent of the existence of the motive over $\QQ$.

This article addresses the next natural step: the construction of a $p$-adic Asai $L$-function when $\Psi$ is not ordinary at $p$. 
Our construction is also independent of the conjectural Asai motive over $\QQ$.
We attach a $p$-adic Asai $L$-function with unbounded denominators  $L^{\mathrm{As}}_{p}(\Psi)$ to a $p$-non-ordinary small slope cuspidal Bianchi modular form $\Psi$ such that it interpolates critical values of the Asai $L$-function associated with $\Psi$. 
This construction generalizes the result of Loeffler--Williams \cite[Theorem 7.5]{CW} from the $p$-ordinary case to the $p$-non-ordinary case.

Roughly, Loeffler--Williams constructed a $p$-adic measure $L^{\mathrm{As}}_{p}(\Psi)$ using Asai--Eisenstein elements only for $j=0$. 
Since they work in the $p$-ordinary setting, by a deep theory of Kings implies that this measure is independent of the twist $j$.
In particular, for a Dirichlet character $\theta$ of $p$-power conductor, by integrating $x^{j}\theta(x)$ against $L^{\mathrm{As}}_{p}(\Psi)$ recoveres the special $L$-value $L^{\mathrm{As}}(\Psi,\overline{\theta},j+1)$. 
See \cite[Propositions 5.5, 6.3]{CW} for more details.
We will briefly outline their work in Section \ref{prevwork}.
The method of using the only Asai--Eisenstein element for $j=0$ breaks down in $p$-non-ordinary setting.
 
To circumvent this issue, we use the approach taken by Amice--Velu in \cite{AV} and Vi\v sik in \cite{Vishik}: interpolation of twists using polynomials. 
In this article, we obtain the $p$-adic Asai $L$-function $L^{\mathrm{As}}_{p}(\Psi)$ by constructing certain novel family of polynomials $P_{r,j}(T)\in E[T]$ of degree $p^{r-1}$, where $\Psi$ is a $p$-non-ordinary weight $(k,k)$ cuspidal Bianchi modular form, $r\in\ZZ_{\geq 1}$ and $0\leq j \leq k$.
These polynomials are constructed by pairing an appropriate modular symbol associated to $\Psi$ with the \emph{Asai-Eisenstein elements} constructed by Loeffler--Williams in \cite[Sections 3, 4]{CW}.
Furthermore, if we evaluate these polynomials at some special elements, we obtain the critical $L$-values $L^{\mathrm{As}}(\Psi, \overline{\theta}, j+1)$ with some explicit factor.
The polynomials $P_{r,j}$ satisfy certain norm and congruence properties.
To prove the congruence properties of polynomials, we state and prove some new congruence results associated to Asai--Eisenstein elements.
Then, by applying the techniques developed by Amice--V\'elu, Vi\v sik, Perrin-Riou, and B\"uy\"ukboduk--Lei to these polynomials, we construct the $p$-adic distribution with unbounded coefficients and some growth, that is, an element in the distribution space $\HH_{E,v_{p}(a_{p})}(\GG)$, where $v_{p}(a_{p})$ is the slope of the '$p$-th' Fourier coefficient.
We hope that one may use the $p$-adic distribution $L^{\mathrm{As}}_{p}(\Psi)$ to formulate and prove the Harron--Pottharst \cite{HP} style Iwasawa main conjecture over the distribution space, relating the ideal generated by this distribution over the distribution space to a certain Selmer complex. 
We briefly outline the construction of these polynomials and congruences in Subsection \ref{prj}.

We further study this $p$-adic Asai $L$-function. 
In particular, in Section \ref{dec}, we construct Pollack, Sprung, and Lei--Loeffler--Zerbes style signed $p$-adic Asai $L$ functions using the techniques developed by the author in \cite{MD}.
These signed $p$-adic Asai $L$-functions may be used to formulate signed Iwasawa main conjectures.

\subsection{The main result}

Throughout the article, $p$ is an odd rational prime and $k\in\ZZ_{\geq 0}$. 
Let $F=\QQ(\sqrt{-D})/\QQ$ be an imaginary quadratic field, $\OO_{F}$ be its ring of integers, and let $\sigma$ be the complex conjugation map.
Let $\Psi$ be a cuspidal Bianchi eigenform over $F$ of level $\frkn$ (i.e. $U_{F,1}(\frkn)$) and weight $(k,k)$. 
See Definition \ref{congdef} for $U_{F,1}(\frkn)$.
The twisted Asai $L$-function of $\Psi$ with a Dirichlet character $\theta$ of conductor $m$ is defined by
\[L^{\mathrm{As}}(\Psi,\theta,s) \coloneqq L^{(mN)}(\theta^{2}\epsilon_{\Psi}|_{(\ZZ/N\ZZ)^\times}, 2s-2k-2)\cdot \sum_{\substack{
            n\geq 1,\\ (m,n)=1}}
        c(n\Of,\Psi)\theta(n)n^{-s},\]
 where $(N)=\frkn\cap\ZZ$, $\epsilon_{\Psi}|_{(\ZZ/N\ZZ)^\times}$ is the restriction to $(\ZZ/N\ZZ)^\times$ of the nebentypus $\epsilon_{\Psi}$ of $\Psi$, $L^{(mN)}(-)$ is the Dirichlet $L$-function with its Euler factors at the primes dividing $mN$ removed, and $c(\frkm, \Psi)$ denotes the Hecke eigenvalue of $\Psi$ at the integral ideal $\frkm$.

We assume that the level $\frkn$ of $\Psi$ is divisible by all the primes above $p$. 
Assume that $\Psi$ is $p$-non-ordinary. More precisely, the $U_{p}$ Hecke eigenvalue $a_{p}$ of $\Psi$ is not a $p$-adic unit, i.e.,$v_{p}(a_{p})>0$, where $v_{p}$ is a $p$-adic valuation such that $v_{p}(p)=1$.
We furthermore assume that $\Psi$ has \emph{a small slope} at $p$, that is, $v_{p}(a_{p})< k+1$.
Let $E/\Qp$ be a finite extension large enough to contain $F$ and all the Hecke eigenvalues of $\Psi$.
We fix some notations for the rest of the article.
Fix a topological generator $u$ of $1+p\Zp$.
Let $\GG=\Gal(\Qp(\mu_{p^\infty})/\Qp)\cong\Delta\times\Zp$ be the Galois group of cyclotomic extension of $\Qp$ by adjoining $p$-power roots to $\Qp$. 
Let $\GG_1$ be a subgroup of $\GG$ such that $\GG_{1}\cong \GG/\Delta\cong\Zp$.
Fix a topological generator $\gamma_{0}$ of $\GG_{1}$.
Let $\chi$ be the $p$-adic cyclotomic character on $\GG$ such that $\chi(\gamma_{0})=u$.
For any real number $w\geq 0$, define the space of \emph{$w$-admissible distributions} over $\GG$ as 
\[\HH_{E,w}(\GG) \coloneqq \left\{ \sum_{\sigma\in\Delta}\sum_{n\geq0} c_{n,\sigma}\cdot\sigma\cdot (\gamma_{0}-1)^{n}: \sup_{w} \dfrac{|c_{n,\sigma}|_{p}}{n^w}<\infty, \forall \sigma\in\Delta\right\}.\]

The main result of this paper is the construction of the $p$-adic distribution $\asaipad(\Psi)\in\HH_{E,v_{p}(a_{p})}(\GG)$ associated with $\Psi$, which interpolates the critical values of Asai $L$-function twisted by a Dirichlet character of conductor $p^r$. 

\begin{thmx}[Theorem \ref{thmasai1}, Theorem \ref{interpolation}]
\label{thma}
For any integer $c>1$ coprime to $6\frkn$, there exists a $v_{p}(a_{p})$-admissible distribution $\asaipad(\Psi)$ with the following interpolation property: for any integer $0\leq j\leq k$ and for any Dirichlet character of conductor $p^{r}>1$ we have 
\[\asaipad(\Psi)(\chi^{j}\theta) = \begin{cases}
\dfrac{*'}{(a_{p})^{r}}\Las(\Psi, \overline{\theta},j+1) &\text{ if } (-1)^{j}\theta(-1)=1,\\[1em]
0 &\text{ if } (-1)^{j}\theta(-1)= 1,
\end{cases}\]
    where $*'$ is an explicit non-zero factor.
\end{thmx}
Note that one can get rid of $c$ under some conditions. For more details, see the Subsection \ref{withoutc}.

\begin{rem}
    By an eigenform, we mean either $\Psi$ is a newform or can be obtained by the $p$\emph{-stabilization} of some Bianchi newform $\Psi'$ of level $\frkn'$ such that $\frkn'$ is coprime to $p$.
\end{rem}

\begin{rem}
    We do not impose any restriction on the fixed odd prime $p$. 
    Our construction works as long as the $p$-adic valuation of the $U_{p}$-eigenvalue is less than $k+1$.
\end{rem}
\subsection{Outline of the construction of $P_{r,j}$}
\label{prj}
We briefly outline the construction of the polynomials $P_{r,j}$ of degree $p^{r-1}$ that are used to construct the $p$-adic $L$-function $\asaipad(\Psi)$.
For $r\in\ZZ_{> 1}$, any integer $0\leq j\leq k$, and $\delta:\Delta\to \OO_{E}^{\times}$, define,
\begin{equation}
    \label{polyeq}
    P^{\delta}_{r,j}(T)=\left \langle \phi^{*}_{\FF}, \hspace{0.5em} (U_p)_{*}^{-r} \frac{1}{A_{j}}\sum_{t\in\Zpr} \prescript{}{c}{\Xi^{k,j}_{p^{r},\frkn,at}}\otimes \delta(t)(1+T)^{\logt} \right\rangle,
\end{equation}
where
\begin{itemize}
\setlength{\itemsep}{2pt}
    \item $\phi^{*}_{\FF}\in\mathrm{H}^{1}_{\mathrm{c}}(\YFss(\frkn), V_{kk}(\OO_{E}))$ is a modular symbol associated to $\FF$,
    \item $\prescript{}{c}{\Xi^{k,j}_{p^r,\frkn,at}}\in\mathrm{H}^{2}(\YFss(\frkn), T_{kk}(\OO_{E}))$ is the \emph{Asai-Eisenstein element} constructed by Loeffler--Williams in \cite{CW},
    \item $\logt$ is a unique integer in $[0,p^{r-1})$ related to $t\in\Zpr$,
    \item $A_j$ is some combinatorial fudge factor and $a\in\OO_F$ such that $a-a^{\sigma}=\sqrt{-D}$,
    \item $\langle,\rangle$ denotes the perfect Poincar\'e duality pairing between $\mathrm{H}^{1}_{\mathrm{c}}(\YFss(\frkn), V_{kk}(\OO_E))$ and $\mathrm{H}^{2}(\YFss(\frkn), T_{kk}(\OO_E))$.
\end{itemize}
We also define a polynomial when $r=1$. See Equation \eqref{trivi}.
Here $\YFss(\frkn)$ is a locally symmetric space of a certain level (see \ref{locsymdef} for the definition).
Here $V_{k,k}$ are $T_{k,k}$ are some specific weight $(k,k)$-coefficient modules. 
The Asai-Eisenstein element $\asaieisen^{k,j}_{p^{r},\frkn,at}$ is related to an Eisenstein series of weight $2k-2j+2$.These elements are 'the Betti counterparts' of the known Euler system elements, such as Beilinson--Flach Euler systems.
They satisfy certain norm-compatibilities like Beilinson--Flach elements. See \cite[Theorem 3.13]{CW} for more details. 

These polynomials satisfy some norm and congruence properties. 
For example, for $0\leq j \leq k$ and $p^{r}>0$, $P_{r,j}$ satisfy,
\begin{equation}
    \label{congeqn}
\sup_{r}\Big{|}\Big{|} p^{(v_{p}(a_{p})-j)r} \sum_{i=0}^{j}(-1)^{i} {j\choose i} P^{\delta}_{r,i}(u^{-i}(1+T)-1)\Big{|}\Big{|}<\infty.
\end{equation}
This property is crucial to construct $\asaipad(\Psi)$.
The more properties of $P_{r,j}$ are described in Lemma \ref{polylem}.
To prove \eqref{congeqn}, we first prove the following key lemma, which is a cohomological version of the interpolation of twists (congruence relation):
\begin{lemx}[Lemma \ref{patch2}]
\label{lemA}
For any $t\in (\ZZ/p^{r})^\times$ and any positive integer $j$, we have
    \begin{align}
        &\sum_{i=0}^{j} t^{(j-i)}(a-a^{\sigma})^{(j-i)}(j-i)!\mathrm{Res}_{p^r}^{p^{jr}} \mom^{(j-i)(j-i)} (\asaieisen^{[i, i]}_{p^{r},\frkn, ta}) \\
        &\in p^{jr}\Ht(\YFss(\frkn), T_{jj}(\OO_{E})) \nonumber.
    \end{align}    
\end{lemx}
For the definition of the moment map $\mom$, see Subsection \ref{momomaps}. The Lemma \ref{patch2} is the Betti cohomology analogue of \cite[Theorem 3.3.5]{LZ_ran}.
Furthermore, we prove another novel lemma :
\begin{lemx}[Lemma \ref{patch3}]
    For any character $\delta:\Delta \to \OO_{E}^\times$, any integer $0\leq j \leq k$, and any integer $r$ such that $p^{r}>1$, we have
    \begin{equation*}
        \sup_{r} \Big{|}\Big{|}p^{-jr} \sum_{i=0}^{j} (-1)^{i}{j\choose i} \sum_{t\in \Zpr} \asaieisen^{k,i}_{p^r, \frkn, at}\otimes t^{-i}\delta(t)(1+T)^{\logt}\Big{|}\Big{|} < \infty.
    \end{equation*}
\end{lemx}
To prove Lemmas \ref{patch2}, \ref{patch3}, we study locally symmetric spaces $\YFs(m,m\frkn)$ of mixed level, where $m\in\ZZ$.
Similar to the methods of Loeffler--Zerbes in \cite{LZ_ran} and Loeffler--Williams in \cite{CW}, for each $j\in\ZZ_{\geq0}$, we construct and study cohomological elements $\prescript{}{c}{\mathcal{Z}^{[j]}_{p^{r},\frkn,a}}\in \mathrm{H}^{2}(\YFs(p^{r},p^{r}\frkn), T_{jj}(\OO_{K}))$.
To construct and study these elements, \emph{moment maps}, \emph{Clebsch--Gordon maps} are used, together with the interplay between them, in a manner similar to those in Kings--Loeffler--Zerbes \cite{KLZ1} and Lei--Loeffler--Zerbes \cite{LLZ_asai}. 
Readers are advised to refer to Sections \ref{patchsec} and \ref{patchsec2} for a detailed description of the polynomials, cohomological elements, and congruences.

\subsection{Signed $p$-adic Asai $L$-functions}
We give an application of the $p$-adic Asai $L$-function $\asaipad(\Psi)$: the decomposition of $p$-adic Asai $L$-functions with unbounded coefficients into signed $p$-adic Asai $L$-functions with bounded coefficients. 
Assume $p$ splits in $F$ as $p\OO_{F}=\PP\PPP$.
See Section \ref{decomp} for the setting for the theorem mentioned below.
We state a few of them here for the convenience of readers:
\begin{enumerate}
    \item $\Psi$ is a cuspidal Bianchi eigenform of level $\mathcal{N}$, where $\mathcal{N}\subset\OO_F$ is an ideal coprime to $p$, trivial nebentypus, and weight $(k,k)$. Furthermore, $\Psi$ is $\PP$-non-ordinary and $\PPP$-ordinary.
    \item For $\qqq\in\pset$, let $a_{\qqq}$ be the $T_{\qqq}$-eigenvalue, and $\alpha_{\qqq}\neq \beta_{\qqq}$ be the roots of the polynomial $X^{2}- a_{\qqq}X+ p^{k+1}$.
    Assume $v_{p}(a_{\PP})> \left\lfloor \dfrac{k}{p-1} \right\rfloor$.
    \item Choose $\alpha_{\PPP}$ such that $v_{p}(\alpha_{\PPP})=0$ and for $\bullet\in\{\alpha_{\PP}, \beta_{\PP}\}$, let $\tilde{\bullet}=\alpha_{\PPP}\bullet$.
    Let $\Psi^{\tilde{\bullet}}$ be $p$-stabilization of $\Psi$ such that $U_p$-eigenvalue of $\Psi^{\tilde{\bullet}}$ is $\tilde{\bullet}$.
\end{enumerate}
Since $v_{p}(\tilde{\alpha}), v_{p}(\tilde{\beta})<k+1$, we can attach $p$-adic distributions $\asaipad(\FF^{\tilde{\alpha}})$ and $\asaipad(\FF^{\tilde{\beta}})$ to $\FF^{\tilde{\alpha}}$ and $\FF^{\tilde{\beta}}$ respectively using Theorem \ref{thma}.
They satisfy certain growth and interpolation properties.
Then using the methods in \cite{MD}, we prove:
\begin{thmx}[Theorem \ref{thmasai2}]
\label{thmb}
  There exist $\prescript{}{c}{L^{\mathrm{As},\sharp}_{p}}, \prescript{}{c}{L^{\mathrm{As},\flat}_{p}} \in E\otimes\OO_{E}[[\GG]]\cong E\otimes\OO_{E}[\Delta][[T]]$ and a logarithmic matrix $\tilde{\Mbar}\in M_{2,2}(\HH_{E})$ such that 
  \begin{equation}
      \begin{pmatrix}
          \asaipad(\Psi^{\rta})\\[1em]
          \asaipad(\Psi^{\rtb})
      \end{pmatrix} = \tilde{Q}^{-1}\tilde{\Mbar}\begin{pmatrix}
          \prescript{}{c}{L^{\mathrm{As},\sharp}_{p}}\\[1em]
          \prescript{}{c}{L^{\mathrm{As},\flat}_{p}}
      \end{pmatrix},
  \end{equation}\\\\
  where $\tilde{Q}=\begin{pmatrix}
      \rta & -\rtb\\[0.5em]
      -\alpha^{2}_{\PPP}p^{k+1} & -\alpha^{2}_{\PPP}p^{k+1}
  \end{pmatrix}$, and $\tilde{\Mbar}$ satisfies properties from Proposition \ref{mbarprop}. 
\end{thmx}
Note that by using the methods of \cite{MD}, we avoid the conjectural $p$-adic Hodge theoretic properties associated with the Galois representation associated with the Asai motive. 

\subsection{Plan of the article}
In Section \ref{locsymdef}, we recall the definitions of locally symmetric spaces in which we are interested. 
We also define Hecke correspondences on these locally symmetric spaces.
We define the Asai $L$-function of Bianchi modular forms in Section \ref{bmfstuff}. Moreover, we recall the definitions of Bianchi modular forms and Bianchi modular symbols in this section. 
In Section \ref{asaieisendef}, we describe the weight $2$ and higher weight Asai-Eisenstein elements constructed by Loeffler--Williams in \cite{CW}.
The construction of these elements is closely related to the elements constructed in \cite{LLZ_asai}.
One can think of the Asai-Eisenstein elements as a Betti counterpart of known Euler system elements like Beilinson--Flach elements. 
We do some explicit calculations in Section \ref{patchsec} similar to \cite[Theorem 6.2.4]{KLZ1} and \cite[Theorem 8.1.4]{LLZ_asai}. 
We also prove some interpolation and congruence theorems, which are Betti analogues of the theorems in \cite[Section 3.3]{LZ_ran}.
In Section \ref{patchsec2}, we prove Theorem \ref{thma} using the interpolation of twists method of Amice--Velu, Perrin-Riou, and B\"uy\"ukboduk--Lei. 
In the last section, we first recall the definition and construction of logarithmic matrices briefly. Then we prove Theorem \ref{thmb}

\section{Locally symmetric spaces}
\label{locsymdef}
\subsection{Setup and notations}
Let $F/\QQ$ be an imaginary quadratic field of discriminant $-D$ and denote its ring of integers by $\OO_F$.
Let $\adeles$ denote the adele ring of $F$, and the finite adeles are denoted by $\adeles^{f}$. 
Define $\widehat{\OO_{F}}\coloneqq \OO_{F}\otimes \widehat{\ZZ} = \OO_{F}\otimes\prod_{p}\Zp$.

Let $\frkn\subset\OO_F$ be an ideal such that $\frkn$ is divisible by all the primes above $p$.
Throughout this article, we assume $\frkn$ is small enough such that the locally symmetric space attached to $\frkn$ is smooth.
Let $N\in\ZZ$ such that $(N)=\ZZ \cap\frkn$ in $\ZZ$.

Let $\hh=\{z\in\CC \mid \mathrm{Im}(z)>0\}$ be the usual upper-half plane with $\gl(\RR)$-action given by M\"obius transformations. 
Define \textit{the upper-half space} or \textit{hyperbolic 3-space} to be $\hh_{3}=\{(z,t)\in \CC\times \RR_{>0}\}$ with $\gl(\CC)$ action given by
\[\begin{pmatrix}
    a & b \\ c & d
\end{pmatrix}\cdot(z,t) = \left(\dfrac{(az+b)\overline{(cz+d)}+a\overline{c}t^{2}}{\mid cz+d\mid^2+\mid c\mid^{2}t^2}, \dfrac{\mid ad-bc \mid t}{\mid cz+d\mid^2+\mid c\mid^{2}t^2}\right).\]
We embed $\hh\hookrightarrow\hh_{3}$ via $x+iy \mapsto (x,y)$, which is compatible with the actions of $\gl(\RR)$ on both sides.
\subsection{Algebraic groups and locally symmetric spaces}
The primary references of this section are \cite{LLZ_asai} and \cite{CW}.

Let
    \begin{align*}
        G=\mathrm{Res}_{F/\QQ}\gl, \hspace{5mm}G^{*}=G\times_{D}\mathbb{G}_m,
    \end{align*}
be the algebraic groups, where where $D\coloneqq \mathrm{Res}_{F/\QQ}\mathbb{G}_{m}$ and the map $G \to D$ is determinant.

\begin{defi}[Locally symmetric spaces]
    For open compact subgroups $U_{\QQ}\subset \gl(\mathbb{A}^{f}_{\QQ})$, $U\subset G(\mathbb{A}^{f}_{\QQ})=\gl(\mathbb{A}^{f}_{F})$, and $U^{*}\subset G^{*}(\mathbb{A}^{f}_{\QQ})=\{g\in\gl(\mathbb{A}^{f}_{F})\mid \det(g)\in(\mathbb{A}^{f}_{\QQ})^\times\}$, the corresponding locally symmetric spaces are defined as
    \begin{align*}
        \YQ(U_{\QQ}) = \gl(\QQ)_{+} \backslash [\gl(\mathbb{A}^{f}_{\QQ}) \times \hh] / U_{\QQ}, \\[1em]
        \YF(U) = \gl(F) \backslash [\gl(\adeles^{f}) \times \hh_{3}] / U, \\[1em]
        \YFs(U^*) = G^{*}(F)_{+} \backslash [G^{*}(\mathbb{A}^{f}_{\QQ}) \times \hh_{3}] / U^*,
    \end{align*}
    where $G^{*}(F)_{+} = \{g \in G^{*}(F) \colon \det(g)>0 \}$.
\end{defi}

\subsection{Locally symmetric spaces of mixed levels}
We define the specific level groups and locally symmetric spaces of a particular interest in this article. 
\begin{defi}[Congruence subgroups]
\label{congdef}
    Let $K$ be either $F$ or $\QQ$, and $\frkm,\mathfrak{n}$, and $\mathfrak{a}$ be the ideals in $\OO_K$.
    Define
    \begin{align*}
        U_{K}(\frkm,\mathfrak{n})\coloneqq \left\{ B\in \gl(\OO_K \otimes \widehat{\ZZ}) \colon B \equiv I_{2} \mod \begin{pmatrix}
        \frkm & \frkm \\ \mathfrak{n} & \mathfrak{n}
    \end{pmatrix}\right\}, \\[0.8em]
    U_{K}(\frkm(\aaa), \mathfrak{n}) \coloneqq \left\{ B \in \gl(\OO_K \otimes \widehat{\ZZ}) \colon B \equiv I_{2} \mod \begin{pmatrix}
        \frkm & \frkm\aaa \\ \mathfrak{n} & \mathfrak{n}
    \end{pmatrix}\right\}.
    \end{align*}
\end{defi}
We write $Y_{K}(\frkm,\mathfrak{n})\coloneqq Y_{K}(U(\frkm,\mathfrak{n}))$ and similarly $Y_{K}(\frkm(\aaa),\mathfrak{n})$.

Furthermore, let
\begin{equation*}
    U^{*}_{F}(\frkm, \mathfrak{n})\coloneqq U_{F}(\frkm,\mathfrak{n}) \cap G^{*},  \hspace{1.5em}
    U^{*}_{F}(\frkm(\aaa), \mathfrak{n})\coloneqq U_{F}(\frkm(\aaa),\mathfrak{n}) \cap G^{*}.
\end{equation*}
We write $Y^{*}_{F}(\frkm,\mathfrak{n})\coloneqq Y^{*}_{F}(U(\frkm,\mathfrak{n}))$ and $Y^{*}_{F,1}(\mathfrak{n})\coloneqq Y^{*}_{F}((1),\mathfrak{n})$, i.e., $Y^{*}_{F}(\frkm,\mathfrak{n})$ for $\frkm=1$.
 
Let $\frkn\subset\OO_F$ be an ideal and $N$ is an integer such that $(N)=\frkn \cap\ZZ$.
We are interested in the following locally symmetric spaces:
\begin{enumerate}
    \item The usual open modular curve $Y_{\QQ,1}(N)$ of level $\Gamma_{1}(N)$. It has only one connected component isomorphic to $\Gamma_{1}(N)\backslash \hh$.

    \item Another (mixed level) modular curve which we are interested in is $Y_{\QQ}(m, mN)$, for any $m\in\ZZ_{\geq 0}$.
    
    \item The space $\YFss(\frkn)$ (which will appear so many times later). This space also has a single connected component isomorphic to $\Gamma^{*}_{F,1}\backslash\hh_{3}$, where 
    \[\Gamma^{*}_{F,1}\coloneqq \left\{ \begin{pmatrix}
        a & b \\ c & d
    \end{pmatrix}\in \mathrm{SL}_{2}(\Of) \colon c\equiv 0, a\equiv d \equiv 1 \mod \frkn \right\}.\]
    \item We are also interested in (mixed level) $Y^{*}_{F}(m,m\frkn)$, where $m\in\ZZ_{\geq 1}$. The space $Y^{*}_{F}(m,m\frkn)$ is not connected and has connected components indexed by group $(\ZZ/m\ZZ)^\times$, since the component group of $Y^{*}_{F}(m,m\frkn)$ is $\QQ^{\times}\backslash\mathbb{A}^{\times}_{\QQ}/\RR_{>0}\det(U^{*}(m,m\frkn))\cong (\ZZ/m\ZZ)^\times$. The identity component $Y^{*}_{F}(m, m\frkn)^{(1)}$ is isomorphic to $\GG^{*}_{F}(m, m\frkn)\backslash\hh_{3}$, where
    \[\GG^{*}_{F}(m,m\frkn)=\left\{ \begin{pmatrix}
         a & b\\ c& d
    \end{pmatrix} \in \mathrm{SL}_{2}(\OO_{F}) : \begin{matrix}
        a \equiv 1 \mod m, & b\equiv 0 \mod m\\c \equiv 0 \mod m\frkn, & d\equiv 1 \mod m\frkn 
    \end{matrix}\right\}.\]
    We will describe some explicit computations related to this space in Section \ref{patchsec}.
    
    \item The space $Y_{F,1}(\frkn)$. Since $\det(U_{F,1}(\frkn))=\widehat{\Of}^{\times}$, $Y_{F,1}(\frkn)$ has $h_{F}$ connected components, where $h_{F}$ is the class number of $F$. The identity component is isomorphic to $\GG_{F,1}\backslash\hh_{3}$, where $\GG_{F,1}(\frkn)\coloneqq \gl(F) \cap U_{F,1}(\frkn).$ 
\end{enumerate}
\begin{rem}
   There are natural maps
    $$\YQQ(N) \xrightarrow{\iota}\YFss(\frkn) \xrightarrow{j} Y_{F,1}(\frkn)$$
    induced by the natural maps $\hh \hookrightarrow\hh_{3}$ and $\gl(\mathbb{A}^{f}_{\QQ}) \hookrightarrow G^{*}(\mathbb{A}^{f}_{\QQ}) \hookrightarrow G^{*}(\mathbb{A}_{\QQ}^{f})$.
\end{rem}

\begin{prop}
If $\frkn$ is divisible by some integer $q\geq 4$, then $\YFss(\frkn)$ is a smooth manifold, and
\[\iota \colon \YQQ(N) \hookrightarrow \YFss(\frkn)\]
is an injective map and hence a closed immersion.
Moreover, let $m\in\ZZ$ be a positive integer, and if $m\frkn$ is divisible some integer $\geq4$, then
\[\iota \colon Y_{\QQ}(m,mN) \hookrightarrow \YFs(m,m\frkn)\]
is an injective map and a closed immersion.
\end{prop}
\begin{proof}
    See \cite[Proposition 2.5]{CW}.
\end{proof}

\begin{assum}
    We will assume the ideal $\frkn$ of $\Of$ is divisible by some integer $q\geq 4$ throughout the article. 
    Due to this assumption, the space $\YFss(\frkn)$ will be a smooth manifold and not a non-smooth orbifold.
\end{assum}

\subsection{Hecke correspondences on the locally symmetric spaces} 

Let $\aaa$ be any square-free ideal of $\OO_K$.
Consider the correspondence diagram
\[\begin{tikzcd}[column sep=small]
& Y_{F}(1(\aaa),\frkn) \arrow[dl, "\pi_2"] \arrow[dr,"\pi_1"] & \\
Y_{F,1}(\frkn)  & & Y_{F,1}(\frkn)
\end{tikzcd}\]
where $\pi_1$ is the natural projection map, and $\pi_2$ is the map given by the right translation action of $\begin{pmatrix}
    \varpi & 0 \\ 0 & 1
\end{pmatrix}$ on $\gl(\mathbb{A}^{f}_{F})$, where $\varpi\in\widehat{\OO_F}$ is integral adele which generates the ideal $\aaa\widehat{\OO_F}$.
We define Hecke correspondences: 
\begin{align*}
    (T_{\aaa})_{*}\coloneqq (\pi_{2})_{*}\circ(\pi_1)^{*},\\
    (T_{\aaa})^{*}\coloneqq (\pi_{1})_{*}\circ(\pi_2)^{*}
\end{align*}
as correspondences on $Y_{F,1}(\frkn)$.
When $\aaa \mid \frkn$, we write $(U_{\aaa})^{*}$ and $(U_{\aaa})_*$ instead of $(T_{\aaa})^{*}$ and $(T_{\aaa})_*$ respectively. 
We can extend these definitions to non-squarefree $\aaa$ in the usual way.

\begin{rem}
    We can use the same construction to define Hecke correspondences on the mixed level locally symmetric space $Y_{F}(\frkm,\frkn)$, but it will not be independent of the choice of generator $\varpi$ of $\aaa$, it will depend on the class of $\varpi \mod 1+\frkm\widehat{\OO_F}$.
    We will use this when $\aaa$ is generated by some $a\in\ZZ_{>0}$, and in that case we will take $\varpi=a$.
\end{rem}
For $m\in\ZZ_{>0}$, $(T_{m})^*= (T_{(m)})^{*}$. Same for $(T_{m})_{*}, (U_{m})^{*},$ and $(U_{m})_{*}$.
From the above remark, it makes sense to define the Hecke operators $(T_{m})^{*}$ and $(T_{m})_*$, for $m\in \ZZ_{>0}$, on the mixed level symmetric spaces $Y^{*}_{F}(\frkm, \frkn)$.  
Moreover, $(T_\aaa)_*$ and $(U_\aaa)_{*}$ are the transpose of $(T_{\aaa})^{*}$ and $(U_{\aaa})^{*}$ respectively with respect to Poincar\'e duality.

\section{Preliminaries related to the Asai $L$-function of Bianchi modular forms}
\label{bmfstuff}
\subsection{Bianchi modular forms}
We recall definitions of Bianchi modular forms and their Fourier expansions.  
For an integer $n\geq 0$, and a ring $R$, define $V_{n}(R)$ to be the space of homogeneous polynomials of degree $n$ in two variables $X,Y$ with coefficients in $R$. 
The space $V_{n}(R)$ can also be described as $\sym^{n}(R^{2})$.
The group of matrices $\gl(R)$ acts on $V_{n}(R)$ both from the right and left, and we denote the corresponding space by $V^{(r)}_{n}(R)$ or by $V^{(\ell)}_{n}(R)$, respectively. 
See \cite[Section 2A]{CW} for the precise definitions.

Let $U$ be an open compact subgroup of $\gl(\adeles^{f})$.
Then for any integer $k\geq 0$, there is a finite dimensional $\CC$-vector space $S_{k,k}(U)$ of \textit{cuspidal Bianchi modular forms} over $F$ of weight $(k,k)$ and level $U$, which are functions
\[\Psi : \gl(F) \backslash \gl(\adeles) / U \to V_{2k+2}^{(r)}(\CC),\]
satisfying certain transformations under right translation by the group $\CC^{\times}\cdot\mathrm{SU}_{2}(\CC),$ and satisfying certain growth conditions and harmonicity.
Fruthermore, the Fourier-Whittaker expansion of $\Psi$ is
\[\Psi\left(\begin{pmatrix}
    \textbf{y} & \textbf{x}\\
    0 & 1
\end{pmatrix}\right) = |\textbf{y}|_{\adeles} \sum_{\zeta\in F^\times} W_{f}(\zeta\textbf{y}_{f}, \Psi)W_{\infty}(\zeta\textbf{y}_{\infty})e_{F}(\zeta\textbf{x}), \]
where the Kirillov function $W_{f}(-, \Psi)$ is a locally constant function on $(\adeles^{f})^\times$ with support contained in a compact subset of $\adeles^f$, $W_{\infty} \colon \CC^{\times} \to V_{2k+2}(\CC)$ be the real analytic function defined in \cite[1.2.1(v)]{Will}, and $e_{F}\colon \adeles/F \to \CC^{\times}$ denote the unique continuous character such that its restriction to $F\otimes \RR$ is 
$x_{\infty} \mapsto e^{2\pi i\mathrm{Tr}_{F/\QQ}(x_{\infty})}$ (Also look at \cite[Theorem 2.9]{CW}).

For $U=U_{F,1}(\frkn)$, the Kirillov function $W_{f}(-, \Psi)$ is supported in $\mathcal{D}^{-1}\widehat{\Of}$, where $\mathcal{D}=(\sqrt{(-D)})$ is the different of $F$.
For an ideal $\frkm$ of $\Of$, we define a coefficient $c(\frkm,\Psi)$ as the value $W_{f}(\textbf{y}_{f},\Psi)$ for any $\textbf{y}_f$ generating the fractional ideal $\mathcal{D}^{-1}\frkm\widehat{\Of}$.
The space $S_{k,k}(U_{F,1}(\frkn))$ has an action of commuting Hecke operators $(T_{\frkm})^*$ for all ideal $\frkm$. 
Moreover, if $\Psi$ is an eigenvector for all these operators, normalized such that $c(1,\Psi)=1$, then the $\frkm$-th Hecke eigenvalue of $\Psi$ is $c(\frkm,\Psi)$. 

\subsection{The Asai $L$-function of a Bianchi modular form}
\label{asailde}
For all $d\in (\Of/\frkn)^\times$, the space $S_{k,k}(U_{F,1}(\frkn))$ has an action of diamond operators $\langle d\rangle$.
On any eigenform $\Psi$, they act via a character $\epsilon_{\Psi}\colon (\Of/\frkn)^\times \to \CC^\times$. 
Let $\epsilon_{\Psi}|_{(\ZZ/N\ZZ)^\times}$ denote the restriction of $\epsilon_{\Psi}$ to $(\ZZ/N\ZZ)^\times$.

\begin{defi}[Asai $L$-function of $\Psi$]
\label{def1asai}
    Let $\Psi$ be a normalized Hecke eigenform in $S_{k,k}(U_{F,1}(\frkn))$ and let $\theta$ be a Dirichlet character of conductor $m$. 
    Define the \emph{Asai $L$-function} of $\Psi$ by
    \begin{equation}
        L^{\mathrm{As}}(\Psi,\theta,s)\coloneqq L^{(mN)}(\theta^{2}\epsilon_{\Psi,\QQ}, 2s-2k-2)\cdot \sum_{\substack{
            n\geq 1,\\ (m,n)=1}}
        c(n\Of,\Psi)\theta(n)n^{-s},
    \end{equation}
    where $(N)=\frkn\cap\ZZ$ and $L^{(mN)}(-,s)$ is the Dirichlet $L$-function with its Euler factors at the primes dividing $mN$ removed. 
    If $\theta$ is trivial, we just write $\Las(\Psi,s)$.
\end{defi}
This Asai $L$-function $\Las(\Psi,\theta,s)$ is absolutely convergent for $\mathrm{Re}(s)$ sufficiently large (one can take $\mathrm{Re}(s)>k+3$) and has meromorphic continuation for all $s\in \CC$. See \cite[Section 2E]{CW} for more details.
For $s$ in the half-plane of convergence, $\Las(\Psi,\theta,s)$ can be written as an Euler product
\[\Laspo = \prod_{\ell \text{ prime}} P^{\mathrm{As}}_{\ell}(\Psi,\theta,s),\]
where the polynomial $P_{\ell}^{\mathrm{As}}(\Psi,\theta,s)$ depends only on $\theta(\ell)$ and the Hecke and diamond eigenvalues of $\Psi$ at the primes above $\ell$. 
For simplicity, assume $\theta$ and the nebentypus of $\Psi$ are trivial. 
For primes $\frkl$ of $F$. let $\alpha_{\frkl}$ and $\beta_{\frkl}$ denote the roots of the polynomial 
$X^{2}-c(\frkl,\Psi)X+N(\frkl)^{k+1}.$
Then we have
\begin{equation}
    \label{lfactors}
    \dfrac{1}{P^{\mathrm{As}}_{p}(\Psi,s)} = \begin{cases}
    (1-\alpha_{\PP}\alpha_{\PPP}p^{-s})(1-\alpha_{\PP}\beta_{\PPP}p^{-s})(1-\beta_{\PP}\alpha_{\PPP}p^{-s})(1-\beta_{\PP}\beta_{\PPP}p^{-s}) &\text{ if } p=\PP\PPP,\\[1em]
    (1-\alpha_\PP p^{-s})(1- p^{-2s+2})(1-\beta_\PP p^{-s}) &\text{ if } p=\PP,\\[1em]
    (1-\alpha^{2}_{\PP}p^{-s})(1-p^{-s+1})(1-\beta^{2}_{\PP}p^{-s}) &\text{ if } p=\PP^2.
\end{cases}
\end{equation}
See \cite[Section 3]{Ghate1} for more details.

\begin{rem}
    \label{rem1}
    The Asai $L$-function appearing in Definition \ref{def1asai} is an "imprimitive" $L$-function. 
    We can define a "primitive" Asai $L$-function using automorphic representations attached to Bianchi modular forms.
    Another way to define the Asai $L$-function of a Bianchi modular form $\Psi$ is the $L$-function attached to the tensor induction to $\QQ$ of Galois representation associated with $\Psi$.
    See \cite[Section 2F]{CW} and \cite[Sections 3, 4]{Ghate1}.
\end{rem}

\begin{lem}{\cite[Lemma 2.11]{CW}}
    Let $\Psi$ be a normalized Bianchi eigenform of level $\mathfrak{n}$ coprime to $p$ and let $\Psi'$ be a $p$-stabilization of $\Psi$ of level $p\mathfrak{n}$.
    Let $\alpha=c(p\OO_{F},\Psi')$.
    Then
    \[\Las(\Psi,\theta,s)=\Las(\Psi', \theta,s)\]
    for any non-trivial Dirichlet character $\theta$ of $p$-power conductor, and 
    \[L^{\mathrm{As}}(\Psi', s)=(1-\alpha p^{-s})^{-1}L^{\mathrm{As}, (p)}(\Psi,s).\]
\end{lem} 

\subsection{Bianchi modular symbols}
\label{bms}
For any field extension $F'$ of $F$, we define the left $F'[\mathrm{GL}_{2}(F)]$-module $V_{kk}(F')\coloneqq V_{k}^{(\ell)}(F') \otimes V_{k}^{(\ell)}(F')^{\sigma}$.
The action of $\mathrm{GL}_{2}(F)$ is given by: For any $B\in\mathrm{GL}_{2}(F)$, $B$ acts in the usual way on the first component and via its complex conjugate $B^{\sigma}$ on the second component.
Due to the action of $\gl(F)$, the space $V_{kk}(F')$ gives rise to a local system of $F$-vector spaces on $Y_{F,1}(\frkn)$, which we also denote by $V_{kk}(F')$.

\begin{thm}[Eichler-Shimura-Harder, {\cite[Theorem 2.18]{CW}}]
    \label{eichshim}
    If $\Psi\in S_{k k}(U_{F,1}(\frkn))$ is a normalized Hecke eigenform, we have the following isomorphism of $1$-dimensional $\CC$-vector spaces
    \[S_{k, k}(U_{F,1}(\frkn))[\Psi] \cong \Ho_{\mathrm{c}}(Y_{F,1}(\frkn), V_{kk}(\CC))[\Psi] \cong \Ho(Y_{F,1}(\frkn), V_{kk}(\CC))[\Psi].\]
    induced by a canonical Hecke-equivariant injection: $$S_{k,k}(U_{F,1}(\frkn)) \hookrightarrow \Ho_{\mathrm{c}}(Y_{F,1}(\frkn), V_{kk}(\CC)).$$
\end{thm}

Thus from Theorem \ref{eichshim}, we have $\eta_{\Psi}\in \Ho_{\mathrm{c}}(Y_{F,1}(\frkn),V_{kk}(\CC))$, corresponding to an eigenform $\Psi\in S_{k,k}(U_{F,1}(\frkn))$. 
Fix an eigenform $\Psi\in S_{k,k}(U_{F,1}(\frkn))$, and let $F'/\QQ$ be a finite extension, large enough, containing $F$ and the Hecke eigenvalues of $\Psi$, $\mathfrak{P}$ be a prime of $F'$ above $p$, and $\OO_{F',(\mathfrak{P})}$ be the valuation ring of $F'$ at $\mathfrak{P}$. 
From \cite[Section 2H]{CW}, we can regard $\Ho_{\mathrm{c}}(Y_{F,1}(\frkn), V_{kk}(\OO_{F',(\mathfrak{P})}))$ as an $\OO_{F',(\mathfrak{P})}$-lattice in $\Ho_{\mathrm{c}}(Y_{F,1}(\frkn, V_{kk}(F'))$. 
Moreover, $\Ho_{\mathrm{c}}(Y_{F,1}(\frkn), V_{kk}(\OO_{F',(\mathfrak{P})}))$ is preserved by the action of Hecke operators $(T_{\aaa})^*$ and $(U_\aaa)^*$.

Define
\[\Ho_{\mathrm{c}}(Y_{F,1}(\frkn), V_{kk}(\OO_{F',(\mathfrak{P})}))[\Psi] \coloneqq \Ho_{\mathrm{c}}(Y_{F,1}(\frkn), V_{kk}(F'))[\Psi] \cap \Ho_{\mathrm{c}}(Y_{F,1}(\frkn), V_{kk}(\OO_{F',(\mathfrak{P})})). \]
From \cite[Proposition 2.20]{CW}, we know that there exists a complex period $\Omega_{\Psi}\in\CC^\times$, such that the quotient
    $\phi_{\Psi}\coloneqq \dfrac{\eta_{\Psi}}{\Omega_{\Psi}}$
    forms an $\OO_{F',(\mathfrak{P})}$-basis of $\Ho_{\mathrm{c}}(Y_{F,1}(\frkn), V_{kk}(\OO_{F',(\mathfrak{P})}))[\Psi]$.
\begin{defi}
    We define
    \[\phi^{*}_{\Psi}\coloneqq j^{*}(\phi_{\Psi})\in \Ho_{\mathrm{c}}(\YFss(\frkn), V_{kk}(\OO_{F',(\mathfrak{P})}))\]
    which is the pullback of $\phi_{\Psi}$ under $j:\YFss(\frkn) \to Y_{F,1}(\frkn).$
\end{defi}
This modular symbol $\phi^{*}_{\Psi}$ will be used later to define the $p$-adic Asai $L$-function of $\Psi$.    

From now onward, as described in the Introduction, the ideal $\frkn\subset \OO_{F}$ is divisible by all the primes above $p$.

\begin{rem}
    By an \emph{eigenform} $\Psi$ of weight $(k,k)$ and level $\frkn$, i.e., of level $U_{F,1}(\frkn)$ we mean either $\Psi$ is a \emph{newform} or $\Psi$ is a \emph{$p$-stabilized $p$-regular eigenform} in the sense:
    \begin{itemize}
        \item $\Psi$ is an eigenform and for each $\PP\mid p$ in $F$, $U_{\PP}(\Psi) = c(\PP,\OO_{F})\Psi$, with $c(\PP,\OO_{F})\neq 0$,
        \item there exists an ideal $\mathfrak{M}$ coprime with $p$ and a Bianchi newform $\FF\in S_{(k,k)}(U_{F,1}(\mathfrak{M}))$ such that $\frkn=\mathfrak{M}\prod_{\PP\mid p}\PP$ and $\Psi$ is obtained from $\FF$ by successive $\PP$-stabilization,
        \item for each $\PP\mid p$, the roots of $X^{2}-c(\PP,\FF)X+\epsilon_{\FF}(\PP)N(\PP)^{k+1}$ are distinct, where $\epsilon_{\FF}$ is the nebentypus of $\FF$.
    \end{itemize}
\end{rem}


\section{Asai-Eisenstein elements}
\label{asaieisendef}
In this section, we first recall definitions of modular units and Kato's Siegel units.
After that, we recall the definitions and constructions of weight $k=2$ as well as higher weight $k> 2$ \textit{Asai-Eisenstein elements} from \cite[Sections 3, 5]{CW}. 
These are compatible families of classes in the Betti cohomology of locally symmetric spaces of Bianchi modular forms. 
Note that there is no \'etale cohomology in the Bianchi setting since the Bianchi manifolds, the locally symmetric spaces, are real manifolds of dimension $3$ and hence are not algebraic varieties. 
Loeffler--Williams constructed these elements by pushing forward Kato's Siegel units to the Betti cohomology (under some maps) but using methods similar to the \'etale cohomology setting from \cite{LLZ_asai}. 

\subsection{Modular units and Siegel units}
A \emph{modular unit} on $Y_{\QQ}(U)$ is an element of $\OO(Y_{\QQ}(U))^\times$, that is, a regular function on $Y_{\QQ}(U)$ with no zeros or poles, where $U\subset \gl(\widehat{\ZZ})$ is an open compact subgroup.
\begin{defi}
    For $N\geq 1$, and $c>1$ an integer coprime to $6N$, let 
    \[\prescript{}{c}{g_N} \in \OO(Y_{\QQ, 1}(N))^\times \]
    be the Kato's Siegel unit $\prescript{}{c}{g_{0,1/N}}$, which is defined in \cite{Kato}.
\end{defi}
By an abuse of notation, we use $\cgn$ again for the pullback of this unit to the mixed modular curve $Y_{\QQ}(M, N)$, for any $M\geq 1$. 

\begin{prop}{{\cite[Section 2.11, Proposition 3.11]{Kato}}}
\label{prop1}
The Siegel units are norm compatible, that is, if $N\mid N'$ and $N$ and $N'$ have the same prime divisors, then under the natural projection map
    $ \mathrm{pr}: Y_{\QQ}(M, N') \to Y_{\QQ}(M, N)$
    we have
    \[(\mathrm{pr})_{*}(\prescript{}{c}{g_{N'}}) = \cgn .\]
\end{prop}


\subsection{Weight $2$ Asai-Eisenstein elements}
For a modular unit $u\in \OO(Y_{\QQ, 1}(N))^\times$, one can associate a Betti realization to $u$: $C(u) \in \mathrm{H}^{1}(Y_{\QQ, 1}(N), \ZZ)$. See \cite[Proposition 3.2]{CW}.

\begin{defi}[Betti-Eisenstein class]
Let $\cC_{N}\coloneqq C(\cgn) \in \Ho(Y_{\QQ, 1}(N),\ZZ)$ be the Betti realization of $\cgn$.    
\end{defi}
Proposition \ref{prop1} implies that if $p\mid N$, the classes $\cC_{Np^{r}}$, for $r\geq 0$, are compatible under push-forward, and hence defines a class
\[\cC_{Np^{\infty}} \in \varprojlim_{r} \Ho(Y_{\QQ, 1}(Np^{r}), \ZZ). \]
Let $\frkn$ be an ideal in $\OO_{F}$ divisible by some integer $\geq 4$ and recall that 
$\YFss(\frkn) = \GG^{*}_{F,1}(\frkn) \backslash \hh_{3}.$
\begin{defi}
    Let $m\geq 1$ be an integer and $a\in \OO_F$.
    Consider the map
    \[\kam: \YFss(m^{2}\frkn) \to \YFss(\frkn)\]
    given by the left action of $\begin{pmatrix}
        1 & a/m \\
        0 & 1
    \end{pmatrix} \in \mathrm{SL}_{2}(F)$ on $\hh_3$.
\end{defi}
Since $\frkn$ is divisible by some integer $\geq 4$, we have maps
\[Y_{\QQ,1}(m^{2}N) \xrightarrow{i} \YFss(m^{2}\frkn) \xrightarrow{\kam} \YFss(\frkn),\]
where $(N)=\ZZ\cap\frkn$.

Note that there are isomorphisms
\begin{align*}
    \Ho(Y_{\QQ,1}(m^{2}N),\ZZ) &\cong \mathrm{H}^{\mathrm{BM}}_{1}(Y_{\QQ,1}(m^{2}N),\ZZ), &
    \Ht(\YFss(m^{2}\frkn),\ZZ) &\cong \mathrm{H}^{\mathrm{BM}}_{1}(\YFss(m^{2}\frkn),\ZZ),
\end{align*}
where $\mathrm{H}^{\mathrm{BM}}_{*}$ denotes Borel-Moore homology (homology with non-compact support). See \cite{BM} for the reference.
We define a push-forward map 
\[i_{*}: \Ho(Y_{\QQ,1}(m^{2}N),\ZZ) \rightarrow \Ht(\YFss(m^{2}\frkn),\ZZ),\]
since Borel-Moore homology is covariantly functorial for proper maps.

\begin{defi}[Weight $2$ Asai-Eisenstein elements]
    For $m\geq 1$ integer, $a\in \OO_{F}/m\OO_{F}$, and $c > 1$ integer coprime to $6mN$, define
    \[\asaieisen_{m,\frkn,a} \coloneqq (\kam)_{*}\circ (i)_{*}(\cC_{m^{2}N})\in \Ht(\YFss(\frkn),\ZZ), \]
    and
    \[\cPhi_{\frkn,a}^{r} = \sum_{t \in (\ZZ/p^{r})^{\times}} \asaieisen_{p^{r},\frkn,at}^{} \otimes [t] \in \htt\otimes \ZZ_{p}[(\ZZ/p^{r})^{\times}].\]
\end{defi}

\begin{rem}
We will use another definition for the Asai--Eisenstein element $\asaieisen_{m,\frkn,a}$ later for our patching arguments involving $\YFs(m,m\frkn)$.    
\end{rem}
\begin{thm}[Loeffler--Williams]
\label{thmlw}
\hfill
    \begin{enumerate}
        \item If $\frkn\mid\frkn'$ are two ideals of $\OO_F$ with the same prime factors, then push-forward along the map $Y_{F,1}(\frkn')\to Y_{F,1}(\frkn)$ sends $\cPhi_{\frkn',a}^{r}$ to $\cPhi_{\frkn,a}^{r}$ for any valid choices of $c, a, r$.

        \item Let $r\geq 1$ be an integer, $a$ be a generator of $\OO_{F}/(p\OO_{F}+\ZZ)$, and let
        \[\pi_{r+1}: \htt\otimes \Zp[(\ZZ/p^{r+1})^\times] \rightarrow \htt\otimes \Zp[(\ZZ/p^{r})^\times] \]
        denote the map which is the identity on the first component and the natural quotient map on the second component.
        Then we have 
        \begin{equation}
            \pi_{r+1}(\cPhi_{\frkn, a}^{r+1}) = (U_{p})_{*}\cdot \cPhi_{\frkn, a}^{r},
        \end{equation}
        where the Hecke operator $(U_{p})_{*}$ acts via its action on $\htt$. 
        Similarly, when $r=0$, we have
        \[\pi_{1}(\cPhi_{\frkn, a}^{1}) = ((U_{p})_{*}-1)\cdot \cPhi_{\frkn, a}^{0}.\]
    \end{enumerate}
\end{thm}

\begin{proof}
    See \cite[Lemma 3.12, Theorem 3.13]{CW}.
\end{proof}

We need the following rephrasing of Theorem \ref{thmlw} in terms of $\asaieisen_{p^{r},\frkn,at}$ elements:
\begin{itemize}
    \item If $\frkn'\mid\frkn$, then $\asaieisen_{p^{r},\frkn',at}$ maps to $\asaieisen_{p^{r},\frkn,at}$ along the pushforward map $\YFs(\frkn')\to\YFs(\frkn)$ for all $t\in \Zpr$.
    \item For $r>0$ and for each $t\in \Zpr$, we have
    \[\sum_{\substack{s \in \ZZ/p^{r+1}\\ s\equiv t \mod p^r}}\asaieisen_{p^{r+1},\frkn,as} = (U_p)_{*}\asaieisen_{p^{r},\frkn,at}.\]

    \item For $r=0$, we have
    \[\sum_{t\in (\ZZ/p)^\times}\asaieisen_{p,\frkn,at} = ((U_{p})_{*}-1)\asaieisen_{1.\frkn,a}.\]
    For more details, see the proof of \cite[Theorem 3.13]{CW}.
\end{itemize}

\subsubsection{Another description of weight $2$ Asai-Eisenstein elements}
Let $a\in\OO_F$. 
Consider the composite map
\[Y_{\QQ}(m,mN)\xhookrightarrow{\iota}\YFs(m,m\frkn) \xrightarrow{u_{a}} \YFs(m, m\frkn),\]
where $u_{a}$ is the action of the matrix $\begin{pmatrix}
    1 & -a \\0 & 1
\end{pmatrix}$ 
on $\YFs(m,m\frkn)$. 
Note that $\begin{pmatrix}
    1 & -a \\0 & 1
\end{pmatrix}$
preserves each component of $\YFs(m,m\frkn)$.

\begin{defi}[Zeta elements]
Define $\czeta_{m,\frkn,a}$ to be the image of $\cC_{mN} = C(\prescript{}{c}{g}_{mN}) \in \Ho(Y_{\QQ}(m,mN), \ZZ)$ under the pushforward $(u_{a}\circ\iota)_{*}$,i.e.
\[\czeta_{m,\frkn,a} = (u_{a}\circ\iota)_{*}(\cC_{mN})\in \Ht(\YFs(m,m\frkn),\ZZ).\]
\end{defi}
For $t\in (\ZZ/m\ZZ)^{\times}$, let $\mathrm{proj}_{t}: \YFs(m,m\frkn) \to \YFs(m,m\frkn)^{(t)}$ be the projection map, and $(\mathrm{proj}_{t})_{*} : \Ht(\YFs(m,m\frkn),\ZZ) \to \Ht(\YFs(m,m\frkn),\ZZ)^{(t)}\coloneqq \Ht(\YFs(m,m\frkn)^{(t)},\ZZ)$ be the projection induced by $\mathrm{proj}_{t}$.

Let $\czeta_{m,\frkn,a}(t)$ be the projection of $\czeta_{m,\frkn,a}$ to the direct summand of $\Ht(\YFs(m,m\frkn),\ZZ)$ given by the $t$-th component. In other words, 
\[\czeta_{m,\frkn,a}(t) = (\mathrm{proj}_{t})_{*}(\czeta_{m,\frkn,a}),\]
and hence we get
\[\czeta_{m,\frkn,a} = \sum_{t} \czeta_{m,\frkn,a}(t).\]

We consider the map
\[s_{m} : \YFs(m,m\frkn) \to \YFss(\frkn)\]
given by the action of $\begin{pmatrix}
    m & 0 \\ 0 & 1
\end{pmatrix}$.
This map corresponds to $(z,t) \mapsto (z/m, t/m)$ on $\hh_3$.

\begin{lem}{\cite[Lemma 4.5]{CW}}
The pushforward of $\cC_{m^{2}N}$ along the map $Y_{\QQ,1}(m^{2}N) \to Y_{\QQ}(m,mN),$ given by $z\mapsto mz$ on $\hh$, is $\cC_{mN}$.
\end{lem}

\begin{thm}{\cite[Proposition 4.4]{CW}}
    Let $a\in\OO_F$ be a generator of $\OO_{F}/(p\OO_F + \ZZ)$. 
    We have 
    \[(s_{p^{r}})_{*}(\czeta_{p^{r},\frkn,a}(t)) = \asaieisen_{p^{r},\frkn,ta},\]
    and thus
    \[\cPhi^{r}_{\frkn,a} = \sum_{t\in (\ZZ/p^{r}\ZZ)^{\times}} (s_{p^{r}})_{*}(\czeta_{p^{r},\frkn,a}(t))\otimes [t]. \]
\end{thm}

We will use$\czeta_{p^{r},\frkn, a}$ for interpolating twists later. 
We also fix $a\in\OO_F$ such that it generates $(\OO_F/(p\OO_{F} + \ZZ))$. 

\subsection{Higher weight Asai-Eisenstein elements}
Let $E$ be a finite extension of $\Qp$ such that $F$ embeds into $E$. Let $\OO_E$ be its ring of integers.
For $k\geq 0$ integer, recall $V_{k}(\OO_{E})= \sym^{k}((\OO_{E})^2)$, the left $\OO_E[\gl(\ZZ)]$-module of symmetric polynomials of degree $k$ in 2 variable with coefficients in $\OO_E$.
Consider $T_{k}(\OO_{E})=(V_{k}(\OO_E))^{*}$ i.e. $T_{k}(\OO_E)$ is the module of symmetric tensors of degree $k$ over $(\OO_{E})^2$. 
We then have the $\OO_{E}[\gl(\OO_{F})]$-module $V_{kk}(\OO_{E})=V_{k}(\OO_{E})\otimes (V_{k}(\OO_{E}))^{\sigma}$, where $\gl(\OO_F)$ acts on the first factor via the embedding $\OO_{F} \hookrightarrow \OO_{E}$ and on the second component via its Galois conjugate.
Let $T_{kk}(\OO_{E})=(V_{kk}(\OO_{E}))^{*}$. 
The space $T_{k}(\OO_{E})$ can be viewed as a local system of $\OO_{E}$-modules on $Y_{\QQ,1}(N)$ for any integer $N\geq 4$. 
Similarly, $T_{kk}(\OO_{E})$ gives a local system on $\YFs(U)$ and $Y_{F}(U)$ for sufficiently small $U$.

\subsubsection{Moment maps}
\label{momomaps}
The linear functional dual to the second basis vector of $\OO_{E}^2$ defines a $\GG^{*}_{F,1}(p^{t}\frkn)$-invariant linear functional on $\sym^{k}((\OO_{E}/p^{t})^{2})$ or on $(\sym^{k}((\OO_{E}/p^{t})^{2}))^{\sigma}$ and hence an invariant vector in $T_{kk}(\OO_{E}/p^{t})$. 
This can be seen as a section of the corresponding local system, defining a class 
\[e_{F,k,t} \in \mathrm{H}^{0}(\YFss(p^{t}\frkn), T_{kk}(\OO_{E}/p^t)).\]
Cup-product with $e_{f,k,t}$ defines a \emph{moment map}
\[\mom^{kk}: \varprojlim_{t} \Ht(\YFss(p^{t}\frkn),\ZZ)\otimes \OO_{E} \rightarrow \Ht(\YFss(\frkn), T_{kk}(\OO_{E})).\]
This is the Betti cohomology analogue of the moment maps in the \'etale cohomology of modular curves considered in \cite{KLZ1}.
In the next section, we will describe and use $e_{f, k, t}$ in more detail.

By Theorem \eqref{thmlw}, the family of classes $(\cPhi_{\frkn p^{t}, a}^r)_{t\geq 0}$ is compatible under pushforward, so it is a valid input to the $\mom^{kk}$ after base-extending from $\OO_{E}$ to $\OO_{E}[(\ZZ/p^{r})^\times]$. 

\begin{defi}
    We let $\cPhi_{\frkn, a}^{k,r}$ be the image of the compatible system $(\cPhi_{\frkn p^{t},a}^{r})_{t\geq 0}$ under $\mom^{kk}$.
\end{defi}

More precisely,
\[\cPhi_{\frkn,a}^{k,r} = (\mom^{kk}\otimes \mathrm{Id})((\cPhi_{\frkn p^{t}, a}^{r})_{t\geq 0})\in \Ht(\YFss(\frkn), T_{kk}(\OO_{E}))\otimes \OO_{E}[(\ZZ/p^{r})^\times], \]
and hence, we have
\[\cPhi_{\frkn, a}^{k,r} = \sum_{t\in (\ZZ/p^{r})^\times} \left(\mom^{kk}((\asaieisen_{p^r,\frkn p^{t}, a})_{t\geq 0})\right) \otimes [t].\]

Note that the action of the Hecke operator $(U_{p})_*$ is well defined both on $\Ht(\YFss(\frkn), T_{kk}(\OO_E))$ and on the $\varprojlim_{t} \Ht(\YFss(\frkn p^{t}),\Zp)$. Moreover, the map $\mom^{kk}$ commutes with $(U_{p})_{*}$. 
Hence, we have the following norm-compatibility relation, for any $k\geq 1$,
\[\pi_{r+1}(\cPhi_{\frkn, a}^{k, r+1}) = (U_{p})_{*}\cdot \cPhi_{\frkn, a}^{k,r}.\]

\subsubsection{Relation to the weight $2k$ Eisenstein series}
From \cite[Proposition 5.2]{CW}, for an integer $k\geq 0$, we know that there exists a class $\mathrm{Eis}^{k}_N \in \Ho(Y_{\QQ,1}(N), T_{k}(\QQ))$, whose image under the comparison map (comparison between the Betti cohomology and the deRham cohomology) is the class of the differential form $-N^{k}F^{k+2}_{1/N}dw^{\otimes k}d\tau$. 
Here $F^{k+2}_{1/N}$ is an Eisenstein series of weight $k+2$ that appears in \cite{Kato}.
Via the base-extension, we can consider $\mathrm{Eis}^{k}_{N} \in \Ho(Y_{\QQ,1}(N), T_{k}(\Qp))$.
This class does not generally lie in the lattice $\Ho(Y_{\QQ,1}(N), T_{k}(\Zp))$.
But for any integer $c>1$ coprime with $2, 3, N$, there exists $\ceis^{k}_{N} \in \Ho(Y_{\QQ,1}(N),T_{kj}(\Zp))$ such that
\[\ceis^{k}_{N} = (c^{2}-c^{-k}\langle c \rangle)\mathrm{Eis}^{k}_{N}\]
holds in $\Ho(Y_{\QQ,1}(N),T_{k}(\Qp))$, where $\langle c \rangle$ is the diamond operator acting on $\Ho(Y_{\QQ,1}(N), T_{k}(\Zp))$ (see \cite{Kings1}, and \cite{KLZ1} for details).
Note that, when $k=j$, we have $\prescript{}{c}{\mathrm{Eis}^{0}_{mN}}=\cC_{mN}$. 

\begin{defi}[Clebsch-Gordan map]
    For $j\in\{1,\dots, k\}$We can regard $T_{2k-2j}(\OO_{E})$ as $\mathrm{SL}_{2}(\ZZ)$-invariant submodule of the $\mathrm{SL}_{2}(\OO_F)$-module $T_{kk}(\OO_{E})$, via the \emph{Clebsch-Gordan map}
    \[\CG^{[k,k,j]} : T_{2k-2j}(\OO_{E}) \to T_{kk}(\OO_{E})\]
    normalized as is in \cite{KLZ1}.
\end{defi}

Recall the composition of maps
$$Y_{\QQ}(m,mN) \xhookrightarrow{\iota} \YFs(m,m\frkn) \xrightarrow{u_{a}} \YFs(m,m\frkn).$$
Using this map, we obtain another composition of maps 
\[(u_{a}\circ\iota)_{*}\circ \CG^{[k,k,j]} : \Ho(Y_{\QQ}(m,mN), T_{2k-2j}(\OO_{E}))\to \Ht(\YFs(m,m\frkn),T_{kk}(\OO_E)).\]

\begin{defi}[Twisted Asai-Eisenstein element]
\label{twstdasai}
    Let $\asaieisen_{m,\frkn,a}^{k,j}\in \Ht(\YFss(\frkn), T_{kk}(\OO_{E}))$ be the image of $(u_{a}\circ\iota)_{*}\circ\CG^{[k,k,j]}(\ceis^{2k-2j}_{mN})$ under the restrictions to the identity component followed by $(s_{m})_{*}$. 
    Similarly, $\Xi_{m,\frkn,a}^{k,j}$ is defined, for the analogous element with $E$-coefficients, using $\mathrm{Eis}^{2k-2j}_{mN}$.
\end{defi}

Explicitly, for $t\in (\ZZ/m\ZZ)^\times$,
\[\asaieisen^{k,j}_{m,\frkn,at} \coloneqq (s_{m})_{*}\circ(\mathrm{proj}_{t})_{*}\circ(u_{a}\circ\iota)_{*}\CG^{[k,k,j]}(\ceis^{2k-2j-2}_{mN}),\]
where $\mathrm{proj}_{t}: \YFs(m,m\frkn) \to \YFs(m,m\frkn)^{(t)}$ is the projection map. 

\begin{rem}
    Note that $\Xi^{k,j}_{p^{r},\frkn,at}\in \Ht(\YFss(\frkn), T_{kk}(E))$, and one has the following equality
    \[\Xi^{k,j}_{p^{r},\frkn,at}= p^{jr}\cdot(\kappa_{at/p^{r}})_{*}(\iota_{*}\CG^{[k,k,j]}(\mathrm{Eis}^{2k-2j}_{p^{2r}N})).\]
    This description is convenient to relate this element with special values of the Asai $L$-function.
    See \cite[Lemma 5.4]{CW} for the details.
\end{rem}

\begin{prop}{\cite[Proposition 5.5]{CW}}
For any integer $r\geq 0$, we have
\[\cPhi^{k,r}_{\frkn, a} = \sum_{t\in (\ZZ/p^r)^\times} \asaieisen_{p^{r},\frkn,at}^{k,0} \otimes [t],\]
where the equality takes place in $\Ht(\YFss(\frkn), T_{kk}(\OO_{E}))\otimes \OO_{E}[(\ZZ/p^r)^\times]$.
\end{prop}

\begin{rem}
    \label{remnormcomp}
    The action of the Hecke operator $(U_{p})_*$ is well defined on $\Ht(\YFss(\frkn),T_{kk}(\OO_E))$ and on $\varprojlim_{r}\Ht(\YFss(\frkn p^{r}),\Zp)$.
    Moreover, the maps $\mom^{kk}, \mom^{(k-j)(k-j)}$ commutes with $(U_{p})_*$, and see the proof of Theorem \ref{thm_decomp} for the relation between $\mom^{(k-j)(k-j)}$ and $\CG^{[k,k,j]}$.
    Thus, like Theorem \ref{thmlw}, we have the following norm-compatibility:
for any integer $r>1$,and any integer $0\leq j \leq k$, we have
    \[\pi_{r+1}\left(\sum_{t\in (\ZZ/p^{r+1})^{\times}}\asaieisen^{k,j}_{p^r,\frkn,at}\otimes [t]\right) = (U_{p})_{*}\sum_{t\in\Zpr}\asaieisen^{k,j}_{p^{r}, \frkn, at}\otimes[t].\]
\end{rem}

\subsection{The $p$-adic Asai $L$-function: $p$-ordinary case}
\label{prevwork}
In \cite{CW}, Loeffler--Williams constructed a $p$-adic measure, that is, a $p$-adic $L$-function in the Iwasawa algebra $\Lambda_{E}(\GG)$. They proved:
\begin{thmt}{\cite[Theorem 7.5]{CW}}
    For any integer $c>1$ coprime to $6\frkn$, there exists a $p$-adic $L$-function 
    \[\asaipad(\Psi)\coloneqq \left\langle \phi^{*}_{\Psi}, \hspace{0.7em} \varprojlim_{r} (U_{p})^{-r}_{*} \mathrm{e}_{\mathrm{ord,*}} \cPhi^{k,r}_{\frkn,a} \right\rangle\in\OO_{E}[[\Zp^\times]]\]
    which satisfies the following interpolation property: for any Dirichlet character $\theta$ of conductor $p^r$, and for any integer $0 \leq j\leq k$, we have
    \[\int_{\Zp^\times}x^{j}\theta(x)d\asaipad(\Psi)(x) = \begin{cases}
        (*)\Las(\Psi,\overline{\theta},j+1) &\text{ if } (-1)^{j}\theta(-1)=1,\\
        0 &\text{ if } (-1)^{j}\theta(-1)= -1,
    \end{cases}\]
    where $(*)$ is some non-zero explicit factor.
\end{thmt}
Here $(\mathrm{e}_{\mathrm{ord},*})$ is the Hida's ordinary projector.
Note that, roughly, by King's theory of polylogarithms (\cite{Kings1}), $\asaipad(\Psi)$ is independent of the twist $j$.
Thus, integrating $x^j\theta(x)$ against $d\asaipad(\Psi)$ computes $\Las(\Psi, \overline{\theta}, j+1)$. 
In other words, we have the following equality
\begin{equation}
    \int_{\Zp^{\times}}x^{j}\theta(x) d\asaipad(\Psi)(x) = \int_{\Zp^{\times}} \theta(x) d\prescript{}{c}{L^{\mathrm{As},j}_{p}(\Psi)(x)},
\end{equation}
where 
\[\prescript{}{c}{L^{\mathrm{As},j}_{p}(\Psi)} = \varprojlim_{r}\left\langle \phi^{*}_{\Psi}, \hspace{0.6em} (U_{p})_{*}^{-r} \mathrm{e}_{\mathrm{ord},*}\sum_{t\in\Zpr} \asaieisen^{k,j}_{p^{r},\frkn, at}\otimes [t]\right\rangle.  \]

\section{Cyclotomic twists of Asai-Eisenstein elements}
\label{patchsec}
In this section, we patch cyclotomic twists of Asai-Eisenstein elements using methods analogous to those in \cite{LZ_ran}. 
We first study the hybrid locally symmetric space $\YFs(m,m\frkn)$ in detail.
After that following \cite{LLZ_ann}, \cite{LLZ_asai}, and \cite{CW}, we define higher weight zeta elements $\czeta^{[j]}_{p^{r},\frkn,a}$ which are required for the patching arguments. 
We also use $\czeta^{[j]}_{p^{r},\frkn,a}$ to obtain congruences like those in \cite{LZ_ran} as well as to define $\asaieisen^{j,j}_{p^{r},\frkn,at}$.

\subsection{On mixed level locally symmetric spaces $\YFs(m, m\frkn)$}
Recall that the locally symmetric space $\YFs(m,m\frkn)$ is not connected and has $\hat{\ZZ}^{\times}/\det(U^{*}_{F}(m, m\frkn))\cong (\ZZ/m)^\times$ connected components.
Each connected component can be identified with $\YFs(m,m\frkn)^{(1)} = \GG^{*}_{F}(m, m\frkn)\backslash\hh_3$ by the matrix $\begin{pmatrix}
    t & 0 \\ 0 & 1
\end{pmatrix},$
where $$\GG^{*}_{F}(m, m\frkn)=\left\{ \begin{pmatrix}
         a & b\\ c& d
    \end{pmatrix} \in \mathrm{SL}_{2}(\OO_{F}) : \begin{matrix}
        a \equiv 1 \mod m, & b\equiv 0 \mod m\\c \equiv 0 \mod m\frkn, & d\equiv 1 \mod m\frkn 
    \end{matrix}\right\}$$
    and $t\in (\ZZ/m)^\times$.
 In other words,
 \[\YFs(m,m\frkn)^{(1)} \xrightarrow[\cong]{\begin{pmatrix}
     t & 0 \\ 0 & 1
 \end{pmatrix}} \YFs(m,m\frkn)^{(t)}.\]

 We have the following commutative diagram for each $t\in (\ZZ/m)^\times$
\begin{equation}
\label{commdiag1}
    \begin{tikzcd}[sep=1.5cm]
\YFs(m,m\frkn)^{(1)} \ar{r}{\matt} \arrow[d, "\matjt"]
& \YFs(m,m\frkn)^{(1)}  \arrow[d, "\matjt"] \\
\YFs(m,m\frkn)^{(t)} \arrow[r,"\mata"]
& \YFs(m,m\frkn)^{(t)}
\end{tikzcd}
\end{equation} 
See \cite[Proof of Proposition 4.4]{CW}.
If we identify $\YFs(m,m\frkn)^{(t)}$ with $\GG^{*}_{F}(m,m\frkn)\backslash\hh_{3}$ via $\matjt$, the restriction to this component of the right action of $\begin{pmatrix}
    1 & -a \\ 0 & 1
\end{pmatrix}$ on the adelic symmetric space corresponds to the left action of $\begin{pmatrix}
    1 & ta \\ 0 & 1
\end{pmatrix}$ on $\hh_3$, i.e, the right action of $\begin{pmatrix}
    1 & -ta \\ 0 & 1
\end{pmatrix}$.

Thus, the right action by $\begin{pmatrix}
    1 & -a \\ 0 & 1
\end{pmatrix}$ can be decomposed (using the above commutative diagram) as 
\begin{align*}
    \begin{pmatrix}
    1 & -a \\ 0 & 1
\end{pmatrix} &= \begin{pmatrix}
    t & 0 \\ 0 & 1
\end{pmatrix}^{-1}\cdot\begin{pmatrix}
    1 & -ta \\ 0 & 1
\end{pmatrix}\cdot\begin{pmatrix}
    t & 0 \\ 0 & 1
\end{pmatrix},\\
&= \begin{pmatrix}
    \frac{1}{t} & 0 \\ 0 & 1
\end{pmatrix}\cdot\begin{pmatrix}
    1 & -ta \\ 0 & 1
\end{pmatrix}\cdot\begin{pmatrix}
    t & 0 \\ 0 & 1
\end{pmatrix}.
\end{align*}

Recall $\czeta_{m,\frkn,a} = (u_{a}\circ\iota)(\cC_{mN}) = \sum_{t\in (\ZZ/m)^\times} \czeta_{m,\frkn,a}(t)$, where $\czeta_{m,\frkn,a}(t)=(\mathrm{proj}_{t})_{*}(\czeta_{m,\frkn,a}).$
Moreover, we know
\[\asaieisen_{p^{r},\frkn,ta} = (s_{p^r})_{*}(\czeta_{p^{r},\frkn, a}(t)) = (s_{p^r})_{*}((\mathrm{proj}_{t})_{*}(\czeta_{m,\frkn,a})).\]

Lets explore the relation between $\czeta_{p^{r},\frkn,a}(t)$ and $\asaieisen_{p^{r},\frkn, ta}$ explicitly.
Using the commutative diagram, we can write
\begin{align*}
    \czeta_{m,\frkn,a}(t) &= (\mathrm{proj_{t}})_{*}(\czeta_{m,\frkn,a}),\\[1em] 
    &= (\mathrm{proj_{t}})_{*}((u_{a})_{*}\iota_{*}(\cC_{mN})), \\[1em]
    &= (u_{a})_{*}\left((\mathrm{proj}_{t})_{*}(\iota)_{*}(\cC_{mN})\right),\\[1em]
    &= \left(\begin{pmatrix}
        \frac{1}{t} & 0\\
        0 & 1
    \end{pmatrix}\cdot u_{ta} \cdot \begin{pmatrix}
        t & 0 \\
        0 & 1
    \end{pmatrix}\right)_{*}((\mathrm{proj}_{t})_{*}(\iota)_{*}(\cC_{mN})). 
\end{align*}
If we write $(\mathrm{proj}_{t})_{*}\iota_{*}(\cC_{mN})= \alpha$, then,
\begin{align*}
    \czeta_{m,\frkn,a}(t) &=\left(\begin{pmatrix}
        \frac{1}{t} & 0\\
        0 & 1
    \end{pmatrix}\cdot u_{ta} \cdot \begin{pmatrix}
        t & 0 \\
        0 & 1
    \end{pmatrix}\right)_{*}(\alpha), \\[1em]
    &= \begin{pmatrix}
        t & 0 \\0 & 1
    \end{pmatrix}_{*}(u_{ta})_{*}\begin{pmatrix}
        t^{-1} & 0 \\0 & 1
    \end{pmatrix}_{*}(\alpha),\\[1em]
    &= \begin{pmatrix}
        t & 0 \\0 & 1
    \end{pmatrix}_{*}(u_{ta})_{*}(\alpha^{(1)}),
\end{align*}
where by superscript $(t)$, we mean the element in $\YFs(p^r,p^r\frkn)^{(t)}$.

Hence,
\begin{align*}
    \asaieisen_{m,\frkn, at} &= (s_{p^r})_{*}(\czeta_{m,\frkn,a}(t)), \\[1em]
    &= (s_{p^r})_{*}\left( \begin{pmatrix}
        t & 0 \\0 & 1
    \end{pmatrix}_{*}(u_{ta})_{*}(\alpha^{(1)})\right),\\[1em]
    &= \begin{pmatrix}
        t & 0 \\ 0 & 1
    \end{pmatrix}_{*}\left((s_{p^r})_{*}(u_{ta})_{*}(\alpha^{(1)})\right), (\text{because } \begin{pmatrix}
        t & 0\\ 0 & 1 
    \end{pmatrix}\text{ and }\begin{pmatrix}
        p^{r} & 0 \\ 0 & 1
    \end{pmatrix}\text{ commutes.})\\[1em]
    &= ((s_{p^{r}})_{*}(u_{ta})_{*}(\alpha^{(1)}))^{(t)}. 
\end{align*}
Since we have defined $\czeta_{m,\frkn,a} = (u_{a})_{*}\iota_{*}(\cC_{mN})$, we get
\[\asaieisen_{m,\frkn,at}= (s_{p^{r}})_{*}((\czeta_{m,\frkn,ta}(1))(t)).\]
This description will be helpful in the following sections.

\subsection{Cyclotomic twists calculations}
Let $E/\Qp$ be a finite extension that contains the quadratic imaginary field $F/\QQ$ and $\OO_{F} \hookrightarrow \OO_{E}$.

If $e_{1}, e_{2}$ is the standard basis of $(\OO_E)^2$, then, for $i\in\{1.2\}$ let $e_{i,r}\coloneqq e_{i}  \mod p^r$. 
Then $e_{F,k,r} = e_{2,r}^{[k]} \otimes e_{2,r}^{[k]}$ 
$\in \mathrm{H}^{0}(\YFss(\frkn p^{r}), T_{kk}(\OO_{E}/p^r))$ is a section,
where $e^{[k]}$ denotes $k$-th divided power of $e$. 
Also, we have chosen $e_{2}$ such that 
$$(u_{a})_{*}(e_{2,r}^{[k]} \otimes e_{2,r}^{[k]}) = e_{2,r}^{[k]} \otimes e_{2,r}^{[k]}.$$ 
We observe that$$(u_{a})_{*}(e_{1,r}\otimes e_{1,r}) = (e_{1,r}-a e_{2,r}) \otimes (e_{1,r} - a^{\sigma} e_{2,r}).$$
Recall that we have $s_{p^r}:\YFs(p^r,p^{r}\frkn) \to \YFss(\frkn)$ given by the action of $\begin{pmatrix}
    p^{r} & 0 \\ 0 & 1
\end{pmatrix}$.
By an abuse of notation, write $T_{kk}(\OO_{E})$ for the local system $\OO_E$-modules both on $\YFss$ and $\YFs$.
Then let $$(s_{p^{r}})_{\sharp}: T_{kk}(\OO_E) \to (s_{p^{r}})^{*}(T_{kk}(\OO_E))$$ be the map on local systems (sheaves) given by the action of $\begin{pmatrix}
    p^{r} & 0 \\ 0 & 1
\end{pmatrix}$ on the representation ($\OO_{E}[\gl(\OO_F)]$-module) $T_{kk}(\OO_{E})$.
Thus $e_{1}^{[k]}\otimes e_{1}^{[k]}$ is in the kernel of $(s_{p^{r}})_{\sharp}\mod p^r$. 
In other words, if we consider the following map
$$(s_{p^r})_{\sharp}:T_{kk}(\OO_E/p^{r}) \to (s_{p^{r}})^{*}(T_{kk}(\OO_E/p^r)),$$
then $e_{1,r}^{[k]}\otimes e_{1,r}^{[k]}$ is in the kernel of $(s_{p^r})_{\sharp}$.

\begin{defi}[$\mod p^r$ moment maps]
    We define $\mod p^r$ moment maps as
    \begin{align*}
        \mom^{kk}_{p^r}: \Ht(\YFss(\frkn p^r),\ZZ)\otimes (\OO_{E}/p^r) &\to \Ht(\YFss(\frkn p^r), T_{kk}(\OO_{E}/p^r)), \\
        z &\mapsto z\cup (e_{2,r}^{[k]}\otimes e_{2,r}^{[k]}).
    \end{align*}
\end{defi}

Recall the \emph{Clebsch-Gordan map}: for $j\in\{1,\ldots,k\}$, we have
\[\CG^{[k,k,j]}:T_{2k-2j}(\OO_E) \to T_{kk}(\OO_{E}).\]
By putting $k=j$ in $\CG^{[k,k,j]}$, we get
\[\CG^{[j]}\coloneqq \CG^{[j,j,j]} : (\OO_{E})^{2} \to T_{jj}(\OO_{E}).\]

Let $x,y \in \mathrm{H}^{0}(Y_{\QQ}(p^r, p^{r}N),\Zp)\otimes (\OO_{E}/p^{r})$ be the sections of order $p^r$ such that $\iota^{*}(e_{1,r})$ and $\iota^{*}(e_{2,r})$ agree with the images of sections $x, y$ respectively under the map
\[\mathrm{H}^{0}(Y_{\QQ}(p^r, p^{r}N),\Zp)\otimes (\OO_{E}/p^{r}) \hookrightarrow \mathrm{H}^{0}(Y_{\QQ}(p^r, p^{r}N),\Zp)\otimes \iota^{*}(\OO_{E}/p^{r}).\]
See \cite[Remark 8.1.2]{LLZ_asai} for a similar situation in the Hilbert modular form setting.

\begin{rem}
    Modulo $p^r$, the $\CG^{[j]}$ is defined by the cup-product with
    \[\sum_{i=0}^{j}(-1)^{i}i!(j-i)!e_{1,r}^{[i]}e_{2,r}^{[j-i]}\otimes e_{1,r}^{[j-i]}e_{2,r}^{[i]}.\]
\end{rem}

\begin{thm}
    \label{thm_decomp}
    For any $t\in (\ZZ/p^{r})^\times$, the following diagram commutes:
\[
  \begin{tikzcd}[sep=1.5cm]
\Ho(Y_{\QQ}(p^{r},p^{r}\frkn),\Zp)\otimes(\OO_{E}/p^{r}) \arrow[r,"\cup (\iota^{*}(e_{2,r}^{[j]}))^{\otimes 2}"] \arrow[d, "\CG^{[j]}"] 
& \Ho(Y_{\QQ}(p^{r},p^{r}N), \iota^{*}(T_{jj}(\OO_{E}/p^{r}))) \arrow[d, "t^{j}j!(a-a^{\sigma})^{j}"] \\
\Ho(Y_{\QQ}(p^{r}, p^{r}N), \iota^{*}(T_{jj}(\OO_{E}/p^{r}))) \arrow[d, "\iota_{*}"]
& \Ho(Y_{\QQ}(p^{r}, p^{r}N), \iota^{*}(T_{jj}(\OO_{E}/p^{r}))) \arrow[d, "\iota_{*}"]  \\
\Ht(\YFs(p^{r}, p^{r}\frkn), T_{jj}(\OO_{E}/p^{r})) \arrow[d, "(u_{a})_{*}"]
& \Ht(\YFs(p^{r}, p^{r}\frkn), T_{jj}(\OO_{E}/p^{r})) \arrow[d, "(u_{a})_{*}"]\\
\Ht(\YFs(p^{r}, p^{r}\frkn), T_{jj}(\OO_{E}/p^{r})) \arrow[d, "(s_{p^r})_{*}\circ(\mathrm{proj}_{t})_{*}"]
& \Ht(\YFs(p^{r}, p^{r}\frkn), T_{jj}(\OO_{E}/p^{r})) \arrow[d, "(s_{p^r})_{*}\circ(\mathrm{proj}_{t})_{*}"]\\
\Ht(\YFss(\frkn), T_{jj}(\OO_{E}/p^r)) \arrow[r, "="] & \Ht(\YFss(\frkn), T_{jj}(\OO_{E}/p^r)) 
\end{tikzcd} 
\]    
\end{thm}
\begin{proof}
    The proof is similar to the proofs of \cite[Theorem 6.2.4]{KLZ1} and \cite[Theorem 8.1.4]{LLZ_asai} in the \'etale cohomology setting. For the convenience of readers, we will prove this theorem. 

    We will use a following fact about divided powers:
    \[(A+B)^{[m]} = \sum_{k=0}^{m}A^{[k]}B^{[m-k]}.\]

    For $z\in \Ho(Y_{\QQ}(p^{r}, p^{r}N), \Zp)\otimes (\OO_{E}/p^{r})$, we have
    \begin{align*}
        (\mathrm{proj}_{t})_{*}(u_{a}\circ\iota)_{*} \CG^{[j]}(z)&= (\mathrm{proj}_{t})_{*}\left((u_{a}\circ\iota)_{*}\left(z \cup \iota^{*}\left(\sum_{i=0}^{j} (-1)^{i} (j-i)! e_{1,r}^{[i]}e_{2,r}^{[j-i]} \otimes e_{1,r}^{[j-i]}e_{2,r}^{[i]}\right))\right)\right), \\[1em]
        &= (u_{a})_{*}\left((\mathrm{proj}_{t})_{*}\left(\iota_{*}(z) \cup \sum_{i=0}^{j} (-1)^{i} (j-i)! e_{1,r}^{[i]}e_{2,r}^{[j-i]} \otimes e_{1,r}^{[j-i]}e_{2,r}^{[i]}\right) \right).
    \end{align*}
Now as in the proof of \cite[Proposition 4.4]{CW} and by the calculations involving $\czeta$, if we write $\mathcal{Z}^{[j]}(t)= (\mathrm{proj}_{t})_{*}\left(\iota_{*}(z) \cup \sum_{i=0}^{j} (-1)^{i} (j-i)! e_{1,r}^{[i]}e_{2,r}^{[j-i]} \otimes e_{1,r}^{[j-i]}e_{2,r}^{[i]}\right) \in \Ht(\YFs(p^{r},p^{r}\frkn)^{(t)}, T_{jj}(\OO_{E}/p^r))$, then
\begin{align*}
(u_{a})_{*}(\mathcal{Z}^{[j]}(t)) &= \left(\begin{pmatrix}
        t & 0 \\ 0 & 1
    \end{pmatrix}_{*}(u_{ta})_{*}\begin{pmatrix}
        t^{-1} & 0 \\ 0 & 1
    \end{pmatrix}_{*}\right)(\mathcal{Z}^{[j]}(t)),\\[1em]
    &= \begin{pmatrix}
        t & 0 \\ 0 & 1
    \end{pmatrix}_{*} (u_{ta})_{*}(\mathcal{Z}^{[j]}(1)). \text{ (naively) }
\end{align*}
    Thus,
    \begin{align*}
        &(u_{ta})_{*}\left(\iota_{*}(z)(1) \cup 
        \left(\sum_{i=0}^{j} (-1)^{i} (j-i)! e_{1,r}^{[i]}e_{2,r}^{[j-i]} \otimes e_{1,r}^{[j-i]}e_{2,r}^{[i]} \right)\right) \\
        &= (u_{ta})_{*}((\iota_{*}(z)(1))) \cup \left(\sum_{i=0}^{j} (e_{1,r}-(ta) e_{2,r})^{[i]} e_{2,r}^{[j-i]} \otimes (e_{1,r}-(ta)^{\sigma} e_{2,r})^{[j-i]} e_{2,r}^{[i]}\right)
    \end{align*}
    
    Therefore, 
    \begin{align*}
        & (s_{p^{r}})_{*}(u_{a})_{*}(\mathcal{Z}^{[j]}(t))\\[1em]
        & = \begin{pmatrix}
            t & 0 \\ 0 & 1
        \end{pmatrix}_{*}(s_{p^r})_{*}\left((u_{ta})_{*}(\iota_{*}(z)(1)) \cup t^{j} j! (a-a^{\sigma})^{j} (e_{2,r}^{[j]}\otimes e_{2,r}^{[j]}) + \text{ sum involving }e_{1,r}\right), \\[1em]
        &=\begin{pmatrix}
            t & 0 \\ 0 & 1
        \end{pmatrix}_{*}\left(((s_{p^r})_{*}(u_{ta})_{*}(\iota_{*}(z)(1))) \cup \left(t^{j} j! (a-a^{\sigma})^{j} (e_{2,r}^{[j]}\otimes e_{2,r}^{[j]})\right)\right),
    \end{align*}
    since 
     $e_{1,r}^{[j]}\otimes e_{1,r}^{[j]}$ is in the kernel of $(s_{p^r})_{\sharp}$ and $e_{2,r}^{[j]}\otimes e_{2,r}^{[j]}$ is invariant under $(s_{p^r})_\sharp$.

    Now, if we chase the diagram on the right-hand side, we get the same equation.
\end{proof}

For $0\leq j\leq k$ and $Y\in\{\YFss(\frkn),\YFs(m,m\frkn)\}$, we define $\mom^{(k-j)(k-j)}$ (similar to $\mom^{kk}$)
\begin{align*}
    \mom^{(k-j)(k-j)}: \Ht(Y,T_{jj}(\OO_{E})) \to \Ht(Y,T_{kk}(\OO_{E})),
\end{align*}
such that modulo $p^r$, $\mom^{(k-j)(k-j)}$ is defined by the cup product with the element $e_{2,r}^{[k-j]}\otimes e_{2,r}^{[k-j]}$.
\begin{lem}
    \label{lemmom}
    For all $0\leq j \leq k$, we have
    \[\mom^{(k-j)(k-j)}_{p^r}\circ\mom^{jj}_{p^r} = {k \choose j}^{2} \mom^{kk}_{p^r}.\]
\end{lem}
\begin{proof}
    This is proved in \cite[Lemma 8.2.1]{LLZ_asai}.
    See also \cite[Lemma 6.3.2]{KLZ1}.
\end{proof}

\begin{thm}
    \label{polynomial}
    For $r\geq 1$ and any $t\in (\ZZ/p^r)^\times$, we have, as classes in $\Ht(\YFss(\frkn), T_{kk}(\OO_{E}/p^r))$ (i.e. equality $\mod p^r$), 
    \[\asaieisen^{k,j}_{p^{r},\frkn, at} = t^{j}(a-a^{\sigma})^{j} j! {k\choose j}^{2}\asaieisen^{k,0}_{p^{r},\frkn,at}.\]
\end{thm}
\begin{proof}
    This is \cite[Proposition 5.6]{CW}. 
    We will prove this for the convenience of the readers.
    We know that for any $t\in \Zpr$, by definition,
    \[\asaieisen^{j,j}_{p^{r},\frkn, at} = (s_{p^r})_{*}(\mathrm{proj}_{t})_{*} (u_{a}\circ\iota)_{*}\CG^{[j]}(\cC_{p^{r}N}).  \]

    On the other hand, from \cite[Proposition 5.5]{CW}, we have
    \[\asaieisen^{j,0}_{p^{r},\frkn,at} = \mom^{jj}(\asaieisen_{p^{r},\frkn,at}).\]
    
    Now, modulo $p^r$, $\mom^{jj}_{p^r}$ is the map defined by the cup product with the element $e_{2,r}^{[j]}\otimes e_{2,r}^{[j]}$. 
    Hence Theorem \ref{thm_decomp} implies
    \begin{equation}
    \label{3.1}
        \asaieisen^{j,j}_{p^{r},\frkn,at} = t^{j}(a-a^{\sigma})^{j} j! \asaieisen^{j, 0}_{p^{r}, \frkn,at},
    \end{equation}
    for any $t\in (\ZZ/p^{r})^\times$.

    Hence, applying Lemma \ref{lemmom} to \eqref{3.1}, we get
    \begin{align*}
        \mom^{(k-j)(k-j)}_{p^{r}}(\asaieisen^{j, j}_{p^{r},\frkn,at}) &= t^{j}(a-a^{\sigma})^{j}j!\mom^{(k-j)(k-j)}_{p^r}(\mom^{jj}_{p^r})(\asaieisen_{p^{r},\frkn, at}),\\[1em]
        &= t^{j}(a-a^{\sigma})^{j}j! {k\choose j}^{2}\mom^{kk}_{p^r}(\asaieisen_{p^{r},\frkn,at}),\\[1em]
        &=t^{j}(a-a^{\sigma})^{j}j! {k\choose j}^{2}\asaieisen^{k,0}_{p^{r},\frkn,at}.
    \end{align*}

    Now using the following commutative diagram (for the local systems on $\YFs$ or modules, see \cite[Proposition 5.1.2]{KLZ1} for more details),
    \begin{equation*}
        \begin{tikzcd}[sep=1.5cm]
            (\OO_{E}/p^{r})^{2} \arrow[r, "\CG^{[j]}"] \arrow[d, "\mom^{2k-2j}_{p^{r}}"]
            & T_{jj}(\OO_{E}/p^r) \arrow[d, "\mom^{(k-j)(k-j)}_{p^{r}}"] \\
            T_{2k-2j}(\OO_{E}/p^{r}) \arrow[r, "\CG^{[k,k,j]}"] 
            &T_{kk}(\OO_{E}/p^{r})
        \end{tikzcd}
    \end{equation*}
    we can conclude
    \[\mom^{(k-j)(k-j)}_{p^r}(\asaieisen^{j,j}_{p^{r},\frkn,at}) = \asaieisen^{k,j}_{p^{r},\frkn, at}.\]
    This completes the proof.
\end{proof}

\subsection{Patching arguments}
For $j>0$, like $\czeta_{p^r, \frkn, a}$, we define
\begin{equation*}
\czeta^{[j]}_{p^{r},\frkn,a} = (u_{a}\circ\iota)_{*}\mathrm{CG}^{[j]}(\cC_{p^{r}N}) \in \Ht(\YFs(p^r, p^{r}\frkn), T_{jj}(\OO_{E})).    
\end{equation*}
Therefore, for any $t \in \Zpr$, we get
\begin{equation}
    \czeta^{[j]}_{p^{r},\frkn,a}(t) =  (\mathrm{proj}_{t})_{*}(u_{a}\circ\iota)_{*}\CG^{[j]}(\cC_{p^{r}N}) \in\Ht(\YFs(p^r, p^{r}\frkn)^{(t)}, T_{jj}(\OO_{E})),
\end{equation}
and\begin{equation}
    \czeta^{[j]}_{p^{r},\frkn,a} = \sum_{t \in \Zpr}\czeta^{[j]}_{p^{r},\frkn,a}(t).
\end{equation}
Note that $\CG^{[j]}:(\OO_E)^2 \to T_{jj}(\OO_{E})$ is defined by the cup product with an element $\mathcal{CG}^{[j]}$ such that $\mathcal{CG}\mod p^{r}$ is $e_{1,r}\otimes e_{2,r} - e_{2,r}\otimes e_{1,r}$. 
Let us denote $e_{1,r}\otimes e_{2,r} - e_{2,r}\otimes e_{1,r}$ by $\mathcal{CG}_{r}$. 
\begin{thm}
    \label{patch1}
    For $a\in \OO_{F}/(p\OO_{F}+\ZZ)$, $j\in \ZZ_{>0}$ and any $t\in (\ZZ/p^{r})^\times$, we have the following equality modulo $p^{jr}$
    \begin{align*}
        \label{patch1eq}
        &\sum_{i=0}^{j} t^{(j-i)}(a-a^{\sigma})^{(j-i)}(j-i)!\mom^{(j-i)(j-i)} \mathrm{Res}_{p^{r}}^{p^{jr}}(\czeta^{[i]}_{p^{r},\frkn, a}(t))\\ &= (\mathrm{proj}_{t})_{*}\left(\mathrm{Res}_{p^r}^{p^{jr}}(u_{a}\circ\iota)_{*}(\cC_{p^{r}N})\cup(u_{ta})_{*}\left( (t(a-(a)^{\sigma})e_{2,jr}\otimes e_{2,jr} + \mathcal{CG}_{jr})^{[j]}\right)\right)\\
         &= \mathrm{Res}_{p^{r}}^{p^{jr}}(\czeta_{p^{r},\frkn,a}(t))\cup(u_{ta})_{*} \left( (t(a-(a)^{\sigma})e_{2,jr}\otimes e_{2,jr} + \mathcal{CG}_{jr})^{[j]}\right).
    \end{align*}
    Here, the map $\mathrm{Res}_{p^r}^{p^{jr}}$ is the pullback along (or induced) by natural projection $\YFs(p^{jr}, p^{jr}\frkn) \to \YFs(p^{r},p^{r}\frkn)$.
\end{thm}
\begin{proof}
This proof is similar to the proof of \cite[Proposition 3.3.4]{LZ_ran}. But here we are proving it in the setting of the Betti cohomology and for the algebraic group over $F$ rather than $\QQ$.
Write $\mathrm{Res}$ for $\mathrm{Res}_{p^{r}}^{p^{jr}}$.
From the equation \eqref{commdiag1}, and calculations involving $\czeta_{p^{r},\frkn,a}(t)$, we have, modulo $p^{jr}$,
\begin{align*}
    &\mathrm{Res}(\czeta_{p^{r},\frkn,a}(t))\cup(u_{ta})_{*} \left( (t(a-(a)^{\sigma})e_{2,jr}\otimes e_{2,jr} + \mathcal{CG}_{jr})^{[j]}\right)\\[1em] &= \mathrm{Res}\left((\mathrm{proj}_{t})_{*}(u_{a}\circ\iota)_{*}(\cC_{p^{r}N}))\right)\cup(u_{ta})_{*} \left( (t(a-(a)^{\sigma})e_{2,jr}\otimes e_{2,jr} + \mathcal{CG}_{jr})^{[j]}\right), \\[1em]
    &= \mathrm{Res}\left(\begin{pmatrix}
        t & 0\\ 0 & 1
    \end{pmatrix}_{*}(u_{ta})_{*}((\iota_{*}\cC_{p^{r}N})(1))\right)\cup (u_{ta})_{*}\left( (t(a-(a)^{\sigma})e_{2,jr}\otimes e_{2,jr} + \mathcal{CG}_{jr})^{[j]}\right), \\[1em]
    &= \begin{pmatrix}
        t & 0\\ 0 & 1
    \end{pmatrix}_{*}(u_{ta})_{*}\left(\mathrm{Res}((\iota_{*}\cC_{p^{r}N})(1))\cup \left (t(a-(a)^{\sigma})e_{2,jr}\otimes e_{2,jr} + \mathcal{CG}_{jr}\right)^{[j]}\right),\\[1em]
    &= \begin{pmatrix}
        t & 0\\ 0 & 1
    \end{pmatrix}_{*}(u_{ta})_{*}\left(\mathrm{Res}((\iota_{*}\cC_{p^{r}N})(1))\cup\sum_{i=0}^{j}(t(a-a^{\sigma})e_{2,jr}\otimes e_{2,jr})^{[j-i]}\mathcal{CG}_{jr}^{[i]}\right),\\[1em]
    &=\begin{pmatrix}
        t & 0\\ 0 & 1
    \end{pmatrix}_{*}(u_{ta})_{*}\left(\sum_{i=0}^{j} \mathrm{Res}((\iota_{*}\cC_{p^{r}N})(1))\cup\mathcal{CG}_{jr}^{[i]}\cup(t(a-a^{\sigma})e_{2,jr}\otimes e_{2,jr})^{[j-i]}\right),\\[1em]
    &= \sum_{i=0}^{j}\begin{pmatrix}
        t & 0\\ 0 & 1
    \end{pmatrix}_{*}(u_{ta})_{*}(\mathrm{Res}((\iota_{*}\cC_{p^{r}N})(1))\cup\mathcal{CG}_{jr}^{[i]})\cup(t(a-a^{\sigma})e_{2,jr}\otimes e_{2,jr})^{[j-i]}.
\end{align*}
Note that modulo $p^{jr}$, $\mom^{(j-i)(j-i)}$ is defined by cup product with the element $e_{2,jr}^{[j-i]}\otimes e_{2,jr}^{[j-i]}$. 
By linear algebra of symmetric tensors, we have $(t(a-a^{\sigma})e_{2,jr}\otimes e_{2,jr})^{[j-i]}= (j-i)!t^{(j-i)}(a-a^{\sigma})^{(j-i)}e_{2,jr}^{[j-i]}\otimes e_{2,jr}^{[j-i]}$.

Hence, by the definition of $\CG^{[j]}$ map, we conclude
\begin{align*}
    &\sum_{i=0}^{j}\begin{pmatrix}
        t & 0\\ 0 & 1
    \end{pmatrix}_{*}(u_{ta})_{*}(\mathrm{Res}((\iota_{*}\cC_{p^{r}N})(1))\cup\mathcal{CG}_{jr}^{[i]})\cup(t(a-a^{\sigma})e_{2,jr}\otimes e_{2,jr})^{[j-i]},\\[1em]
    &= \sum_{i=0}^{j} t^{(j-i)}(a-a^{\sigma})^{(j-i)}(j-i)! (\mathrm{Res}((\mathrm{proj}_{t})_{*}(u_{a})_{*}(\iota)_{*}\CG^{[i]}(\cC_{p^{r}N}))\cup (e_{2,jr}^{[j-i]}\otimes e_{2,jr}^{[j-i]}), \\[1em]
    &= \sum_{i=0}^{j} t^{(j-i)}(a-a^{\sigma})^{(j-i)}(j-i)! \mom^{(j-i)(j-i)}\mathrm{Res}(\czeta^{[i]}_{p^{r},\frkn,a}(t)),
\end{align*}
completing the proof.
\end{proof}
Using Theorem \ref{patch1}, we obtain the following key congruence:
\begin{lem}
\label{patch2}
For any $t\in (\ZZ/p^{r})^\times$, we have
    \begin{align}
        &\sum_{i=0}^{j} t^{(j-i)}(a-a^{\sigma})^{(j-i)}(j-i)!\mathrm{Res}_{p^r}^{p^{jr}} \mom^{(j-i)(j-i)} (\asaieisen^{[i, i]}_{p^{r},\frkn, ta}) \\
        &\in p^{jr}\Ht(\YFss(\frkn), T_{jj}(\OO_{E})) \nonumber.
    \end{align}
\end{lem}

\begin{proof}
    The proof is similar to the proof of \cite[Theorem 3.3.5]{LZ_ran}.

    Recall we have defined $\asaieisen^{j,j}_{p^{r},\frkn, at}$ as $(s_{p^{r}})_{*}(\mathrm{proj}_{t})_{*}(u_{a}\circ\iota)_{*}\CG^{[j]}(\cC_{p^{r}N})$. Thus, we have
    \[\asaieisen^{j,j}_{p^{r},\frkn,at} = (s_{p^{r}})_{*}(\czeta^{[j]}_{p^{r},\frkn,at}(t)),\]
    since we have defined $\czeta^{[j]}_{p^{r},\frkn,at}(t)= (\mathrm{proj}_{t})_{*}(u_{a}\circ\iota)_{*}\CG^{[j]}(\cC_{p^{r}N})$.

    From Theorem \ref{patch1}, $\sum_{i=0}^{j} t^{(j-i)}(a-a^{\sigma})^{(j-i)}(j-i)!\mathrm{Res}_{p^r}^{p^{jr}} \mom^{(j-i)(j-i)} (\asaieisen^{[i, i]}_{p^{r},\frkn, ta})$ modulo $p^{jr}$ boils down to
    \[\mathrm{Res}_{p^{r}}^{p^{jr}}\left((s_{p^{r}})_{*}\left( \czeta_{p^{r},\frkn,a}(t) \cup (u_{ta})_{*}\left((t(a-a^{\sigma})e_{2,r}\otimes e_{2,r} + \mathcal{CG}_{r})^{[j]}\right)\right)\right).\]

    We claim that
    \[(s_{p^r})_{*}(u_{ta})_{*}\left((t(a-a^\sigma)e_{2,r}\otimes e_{2,r} + \mathcal{CG}_{r})^{[j]}\right) = 0.\]

    Note that 
    \begin{align*}
        (u_{ta})_{*}(t(a-a^\sigma)e_{2,r}\otimes e_{2,r} + \mathcal{CG}_{r}) &= t(a-a^\sigma)e_{2,r}\otimes e_{2,r} + (u_{ta})_{*}(e_{1,r}\otimes e_{2,r} - e_{2,r}\otimes e_{1,r}), \\[0.5em]
        &= t(a-a^\sigma)e_{2,r}\otimes e_{2,r} + (e_{1,r}-t\cdot a e_{2,r})\otimes e_{2,r} - e_{2,r}\otimes (e_{1,r}- t\cdot a^{\sigma}e_{2,r}),\\[0.5em]
        &= e_{1,r}\otimes e_{2,r} - e_{2,r}\otimes e_{1,r}.
    \end{align*}
   This implies 
   \begin{align*}
       (s_{p^{r}})_{*}((u_{ta})_{*}(t(a-a^\sigma)e_{2,r}\otimes e_{2,r} + \mathcal{CG}_{r})) &= (s_{p^{r}})_{*}(e_{1,r}\otimes e_{2,r} - e_{2,r}\otimes e_{1,r}),\\ 
       &=0, 
   \end{align*}
   since $e_{1,r}\otimes e_{2,r} - e_{2,r}\otimes e_{1,r}$ is in $\mathrm{Ker}((s_{p^{r}})_{\sharp})\mod p^{r}$.

    Because this element is killed by $(s_{p^r})_{*}$ modulo $p^{r}$, its $j$-th tensor power is will be zero after applying $(s_{p^r})_{*}\mod p^{jr}$ and hence we are done.  
\end{proof}

In the next section, we will use Theorem \ref{patch1} to interpolate (patch) different polynomials to obtain the $p$-adic distribution. 

\section{Interpolation of twists and construction of the $p$-adic distribution}
\label{patchsec2}
Recall $-D$ to be the discriminant of $F$. 
We have assumed the level $\frkn$ to be divisible by some integer $\geq 4$.
This is automatic if $p$ is unramified in $F$ and $p\geq 5$. (\cite[Page 1692]{CW}).
Note that we have taken $a\in \OO_F$ such that $a$ generates $\OO_{F}/(p\OO_{F} + \ZZ)$. 
Thus take $a=\dfrac{1}{2}(1+\sqrt{-D})$ if $D\equiv -1 \mod 4$, and $a=\dfrac{1}{2}\sqrt{-D}$ if $D \equiv 0 \mod 4$.
Thus $a-a^{\sigma}=\sqrt{-D}$.
Recall $E/\Qp$ is a finite extension large enough such that $F$ embeds into $E$ and all Hecke eigenvalues are also in $E$.

\subsection{Polynomial setup}
\label{polysetup}
To construct the distribution, i.e., a power series with unbounded coefficients, we use the technique developed by Perrin-Riou (\cite[Subsection 1.2]{PR1}) and by B\"uy\"ukboduk--Lei (\cite[Section 2]{BL1}). 
To use these methods, we need the language of polynomials, or rather polynomials modulo $\omega_{r}(X)=(1+X)^{p^r}-1$. 

Write $\ZZ_{p}^\times \cong \Delta \times (1+p\Zp)$, where $\Delta$ is cyclic group of order $p-1$.
As mentioned in the introduction, fix topological generator $u$ of $1+p\Zp$, i.e., $\overline{\langle u \rangle}=1+p\Zp$.
For any integer $r\geq 1$, define
\[u_{r} \coloneqq u\mod p^{r}.\]
Then $u_{r}\in (\ZZ/p^{r})^\times$.
Let $\varepsilon: \Delta\to \Zp^{\times}$ denote the Teichm\"uller character.
Then for any $x\in \Zp^{\times}$, $\dfrac{x}{\varepsilon(x)}\in 1+p\Zp$.
Now for any $t\in \Zpr$, let $\tilde{t}\in\Zp^\times$ be a lift of $t$.
Then there exists a unique integer $0 \leq m < p^{r-1}$, such that
\[t\equiv \tilde{t}\equiv \varepsilon(\tilde{t})u^{m} \mod p^r.\]
In particular, we get
\[\dfrac{\tilde{t}}{\varepsilon(\tilde{t})}\equiv u^{m}\mod p^{r}\]
and thus $\dfrac{\tilde{t}}{\varepsilon(\tilde{t})} = u_{r}^{m} \in \langle u_{r} \rangle$, since order of $u_{r}$ is $p^{r-1}$.
By writing $\varepsilon_{p}(x)$ for $\varepsilon(x)\mod p$, for $t\in \Zpr$, we define $\logt$ to be a unique integer $0\leq m < p^{r-1}$ such that
\[\logt\coloneqq \dfrac{t}{\varepsilon_{p}(\tilde{t})} = \dfrac{\tilde{t}}{\varepsilon(\tilde{t})}= u_{r}^{m}.\]

Let $\delta: \Delta \to \OO_{E}^\times$ be a group homomorphism. 
Note that $\delta$ is a non-negative integer power of the Teichm\"uller character $\varepsilon$. 
We then have a ring homomorphism
\begin{align*}
    \OO_{E}[\Zpr] &\to \OO_{E}[T]; \\
    [t] &\mapsto \delta(t) (1+T)^{\logt}.
\end{align*}
Note that $\delta(t)$ is an abuse of notation for $\delta(t\mod p)$. 

From now onwards, we will use the polynomial setup.

\subsection{Construction of the $p$-adic $L$-function at non-ordinary primes}
\label{cont}
Define a norm $||\cdot||$ on $E[T]$ as
\[||f||\coloneqq \sup_{|z|_{p}\leq 1}|f(z)|_{p}\]
for any polynomial $f\in E[T]$.

Note that $\Ht(\YFss(\frkn), T_{kk}(\OO_E))\otimes E = \Ht(\YFss(\frkn), T_{kk}(E))$ and hence $\Ht(\YFss(\frkn), T_{kk}(\OO_E))$ is an $\OO_E$-lattice in $\Ht(\YFss(\frkn), T_{kk}(E))$.
Hence it defines a norm $|\cdot|$ on $\Ht(\YFss(\frkn), T_{kk}(E))$.
By an abuse of notation, we define a norm $||\cdot||$ on $\Ht(\YFss(\frkn), T_{kk}(\OO_E))\otimes \OO_{E}[T]$ using $|\cdot |$. 

Like $\cPhi^{k,j,r}_{\frkn,a} = \sum_{t\in\Zpr} \asaieisen^{k,j}_{p^{r},\frkn,at}\otimes [t]$ defined in \cite{CW}, we define polynomials, for $\delta\in \Delta^*$:
\[\sum_{t\in\Zpr} \asaieisen^{k,j}_{p^{r},\frkn, at}\otimes \delta(t)(1+T)^{\logt}\]
lying in $\Ht(\YFss(\frkn), T_{kk}(\OO_E))\otimes \OO_{E}[T]$.

We now use congruences from Lemma \ref{patch2} to prove the following congruence theorem.

\begin{lem}
    \label{patch3}
    For any character $\delta:\Delta \to \OO_{E}^\times$, any integer $0\leq j \leq k$, and $p^r>0$, we have
    \begin{equation}
    \label{patcheq1}
        \sup_{r} \Big{|}\Big{|}p^{-jr} \sum_{i=0}^{j} (-1)^{i}{j\choose i} \sum_{t\in \Zpr} \asaieisen^{k,i}_{p^r, \frkn, at}\otimes t^{-i}\delta(t)(1+T)^{\logt}\Big{|}\Big{|} < \infty.
    \end{equation}
\end{lem}
\begin{proof}
For simplicity, we assume $\delta$ is the trivial character and $r>1$. 
The proof for non-trivial $\delta$ is similar.

We will use the fact, for $0\leq i \leq j \leq k$, we have
\[\mom^{(k-j)(k-j)}_{p^r}\mom^{(j-i)(j-i)}_{p^r} = {k-i \choose j-i}^2 \mom^{(k-i)(k-i)}_{p^r}.\]
We have the following equality $\mod p^r$:
\begin{align*}
    & \sum_{i=0}^{j} (-1)^{i}{j \choose i} \dfrac{1}{(a^{\sigma}-a)^{i}i!{k \choose i}^{2}}\sum_{t \in \Zpr} \asaieisen^{k,i}_{p^{r},\frkn,at}\otimes t^{-i}(1+T)^{\logt}, \\[1em]
    & = \sum_{i=0}^{j} (-1)^{i}{j \choose i} \dfrac{1}{(a^{\sigma}-a)^{i}i!{k \choose i}^{2}}\sum_{t \in \Zpr} \mom^{(k-i)(k-i)}\asaieisen^{i,i}_{p^{r},\frkn,at}\otimes t^{-i}(1+T)^{\logt}, \\[1em]
    &= \sum_{i=0}^{j} (-1)^{i}{j \choose i} \dfrac{1}{(a^{\sigma}-a)^{i}i!{k \choose i}^{2}}\sum_{t \in \Zpr} \dfrac{1}{{k-i \choose j-i}^2}\mom^{(k-j)(k-j)}\mom^{(j-i)(j-i)}\asaieisen^{i,i}_{p^{r},\frkn,at}\otimes t^{-i}(1+T)^{\logt}, \\[1em]
    &= \mom^{(k-j)(k-j)}\left(\sum_{i=0}^{j} (-1)^{i}{j \choose i} \dfrac{1}{(a^{\sigma}-a)^{i}i!{k \choose i}^{2}}\sum_{t \in \Zpr} \dfrac{1}{{k-i \choose j-i}^2}\mom^{(j-i)(j-i)}\asaieisen^{i,i}_{p^{r},\frkn,at}\otimes t^{-i}(1+T)^{\logt}\right).
\end{align*}

We can simplify factorials as
\begin{align*}
    {j\choose i}\times \dfrac{1}{i! {k\choose i}^2} \times \dfrac{1}{{k-i \choose j-i}^2} 
    &= \dfrac{1}{{k\choose j}^2} (j-i)!.
\end{align*}

Thus we get 
\begin{align*}
    &\mom^{(k-j)(k-j)}\left(\sum_{i=0}^{j} (-1)^{i}{j \choose i} \dfrac{1}{(a^{\sigma}-a)^{i}i!{k \choose i}^{2}}\sum_{t \in \Zpr} \dfrac{1}{{k-i \choose j-i}^2}\mom^{(j-i)(j-i)}\asaieisen^{i,i}_{p^{r},\frkn,at}\otimes t^{-i}(1+T)^{\logt}\right)\\[1em]
    &= \mom^{(k-j)(k-j)}\left(\sum_{i=0}^{j} (-1)^{i} \dfrac{1}{{k\choose j}^2} (j-i)!\dfrac{1}{(a^{\sigma}-a)^{i}} \sum_{t\in \Zpr} \mom^{(j-i)(j-i)}\asaieisen^{i,i}_{p^{r},\frkn,at} \otimes t^{-i}(1+T)^{\logt}
\right),\\[1em]
&= \dfrac{\mom^{(k-j)(k-j)}}{{k\choose j}^2}\left(\sum_{i=0}^{j} (j-i)!\dfrac{1}{(a-a^{\sigma})^{i}} \sum_{t\in \Zpr} \mom^{(j-i)(j-i)}\asaieisen^{i,i}_{p^{r},\frkn,at} \otimes t^{-i}(1+T)^{\logt}
\right),\\[1em]
&= \dfrac{\mom^{(k-j)(k-j)}}{{k\choose j}^2}\left(\sum_{t\in\Zpr}\left(\sum_{i=0}^{j} t^{-i}(a-a^{\sigma})^{-i}(j-i)!\mom^{(j-i)(j-i)}\asaieisen^{i, i}_{p^{r},\frkn,at}\right)\otimes(1+T)^{\logt}\right)
\end{align*}

From the Lemma \ref{patch2}, for all $t \in \Zpr$, we know$\mod p^{jr}$,
\[\sum_{i=0}^{j} t^{(j-i)}(a-a^{\sigma})^{(j-i)}(j-i)!\mom^{(j-i)(j-i)}\asaieisen^{i, i}_{p^{r},\frkn,at} = 0.\]
Thus, for any $t\in\Zpr$,
\begin{align*}
    &\sum_{i=0}^{j} t^{(-i)}(a-a^{\sigma})^{(-i)}(j-i)!\mom^{(j-i)(j-i)}\asaieisen^{i, i}_{p^{r},\frkn,at},\\[1em]
    &=\dfrac{1}{t^{j}(a-a^{\sigma})^{j}}\sum_{i=0}^{j} t^{(j-i)}(a-a^{\sigma})^{(j-i)}(j-i)!\mom^{(j-i)(j-i)}\asaieisen^{i, i}_{p^{r},\frkn,at}\\
    &\in C\cdot p^{jr}\Ht(\YFss(\frkn), T_{jj}(\OO_{E})),
\end{align*}
where $C$ is some positive constant independent of $p^r$ related to $(a^{\sigma}-a)^j$. Note that if $p$ is unramified in $F$ then $a-a^{\sigma}$ is a $p$-adic unit and hence $C=1$.

Therefore, using the norm $||\cdot||$ on $\Ht(\YFss(\frkn), T_{kk}(\OO_E))\otimes \OO_{E}[T]$, we deduce
\begin{align*}
    & \Bigg{|}\Bigg{|}\dfrac{\mom^{(k-j)(k-j)}}{{k\choose j}^2}\left(\sum_{t\in\Zpr}\left(\sum_{i=0}^{j} t^{-i}(a-a^{\sigma})^{-i}(j-i)!\mom^{(j-i)(j-i)}\asaieisen^{i, i}_{p^{r},\frkn,at}\right)\otimes(1+T)^{\logt}\right)\Bigg{|}\Bigg{|}\\ &< C'\cdot p^{jr},
\end{align*}
for some constant $C'$ independent of $r$. 
and hence
\begin{equation}
    \label{maineq1}
    \sup\Big{|}\Big{|} p^{-jr}\sum_{i=0}^{j} (-1)^{i}{j \choose i} \dfrac{1}{(a^{\sigma}-a)^{i}i!{k \choose i}^{2}}\sum_{t \in \Zpr} \asaieisen^{k,i}_{p^{r},\frkn,at}\otimes t^{-i}(1+T)^{\logt} \Big{|}\Big{|}<\infty.
\end{equation}
This completes the proof.
\end{proof}

\begin{rem}
    The equation \eqref{patcheq1} can be interpreted as: the sum 
    \[\sum_{i=0}^{j} (-1)^{i}{j \choose i} \dfrac{1}{(a^{\sigma}-a)^{i}i!{k \choose i}^{2}}\sum_{t \in \Zpr} \asaieisen^{k,i}_{p^{r},\frkn,at}\otimes t^{-i}(1+T)^{\logt}\]
    lies in $C'\cdot p^{jr} \Ht(\YFss(\frkn), T_{kk}(\OO_E))\otimes \OO_{E}[T]$, where $C'$ is some positive constant independent of $p^r$.
\end{rem}

Before going forward, we will introduce some notations related to distributions and power series. 
We define Iwasawa algebras $\Lambda_{E}\coloneqq E\otimes\OO_{E}[[T]]$ and $\Lambda_{E}(\Zp^\times)\coloneqq E\otimes \OO_{E}[[\Zp^\times]]\cong E\otimes \OO_{E}[\Delta][[T]]$.

For any real number $w\geq0$, let
\[\HH_{E,w}\coloneqq\left\{\sum_{n=0}^{\infty}c_{n}T^{n}\in E[[T]] : \sup_{n}\dfrac{|c_{n}|_{p}}{n^w}<\infty\right\}\]
be the space of $w$-tempered/admissible distributions.
Note that, $\HH_{E,w}$ is a $\Lambda_E$-module and when $w=0$, $\HH_{E,0}= \Lambda_{E}$.
Similarly, we define
\[\HH_{E,w}(\Zp^\times)\coloneqq \left\{\sum_{\sigma\in\Delta}\sum_{n=0}^{\infty}c_{n,\sigma}\cdot\sigma\cdot T^{n}\in E[\Delta][[T]] : \sup_{n}\dfrac{|c_{n,\sigma}|_{p}}{n^w}<\infty, \forall \sigma\in\Delta\right\}. \]
Let $\Delta^{*}=\mathrm{Hom}_{\mathrm{cts}}(\Delta, \OO_{E}^\times)$ be the group of character and let $e_{\delta} = \frac{1}{|\Delta|} \sum_{d\in\Delta} \delta^{-1}(d)\cdot d\in\OO_{E}[\Delta]$ be the idempotent corresponding to character $\delta\in\Delta^{*}$.
Then
\[\HH_{E,w}(\Zp^\times) \cong \bigoplus_{\delta\in\Delta^{*}}e_{\delta}(\HH_{E,w}(T)),\]
and $e_{\delta}(\HH_{E,w}(T))\cong \HH_{E,w}$ as $\Lambda_E$-modules.
We say $f$ is $O(\log_{p}^w)$ whenever $f\in\HH_{E,w}(T)$ or $f\in \HH_{E,w}(\Zp^\times)$.

We need the following lemma to construct the distribution.
\begin{lem}
\label{patchlemm}
Let $s\geq 0$ and $h\geq 1$ be integers and $0\leq s<h$. For $0\leq j \leq h-1$, let $Q_{r,j}(T)\in E[T]$ be a sequence of polynomials satisfying
\begin{enumerate}
    \item $\sup ||p^{sr} Q_{r,j}(T)||<\infty$,
    \item $Q_{r+1,j}\equiv Q_{r,j} \mod \omega_{r-1}(T)E[T]$
\end{enumerate}
for all positive integers $r$.
Moreover, suppose that
\[\sup_{r}\Big{|}\Big{|} p^{(s-j)r} \sum_{i=0}^{j}(-1)^{i}{j\choose i}Q_{r,i}(u^{-i}(1+T)-1)  \Big{|}\Big{|} < \infty\]
for all $0\leq j\leq h-1$.
Then there exists a unique polynomial $Q_{r}$ of degree $< hp^{r}$ such that
\begin{enumerate}
    \item $Q_{r}(T)\equiv Q_{r,j}(u^{-j}(1+T)-1) \mod \omega_{r}(u^{-j}(1+T)-1)E[T]$,
    \item $\sup ||p^{sr}Q_{r}(T)||<\infty$,
    \item $Q_{r+1}\equiv Q_{r}\mod \prod_{i=0}^{h-1}\omega_{r-1}(u^{-i}(1+T)-1)$,
\end{enumerate}
Moreover, the sequence $(Q_{r})_{r}$ converges to a power series $Q_{\infty}$ such that $Q_{\infty}\in\HH_{E,s}$.
\end{lem}
\begin{proof}
    See \cite[Lemme 1.2.1, Lemme 1.2.2]{PR1} and \cite[Lemma 2.2, Lemma 2.3]{BL1} for the details.
\end{proof}

Now we will construct the distribution using Lemma \ref{patchlemm}.
Let $\Psi$ be a Bianchi eigenform over $F$ of weight $(k,k)$ and level $\frkn$, i.e. of level $U_{F,1}(\frkn)$.
Assume $\Psi$ is $p$-non ordinary and smalll slope, i.e. $v_{p}(a_{p})>0$ and $v_{p}(a_{p})<k+1$, where $a_{p}$ is the $U_p$-eigenvalue.

Recall from Subsection \ref{bms}, that if $F'/\QQ$ is a finite extension obtained by adjoining all the Hecke eigenvalues of $\Psi$ to $F$, and let $\mathfrak{P}$ be a prime above $p$ in $E$, then we defined
$$\phi^{*}_{\Psi} \in \Ho_{\mathrm{c}}(\YFss(\frkn), V_{kk}(\OO_{F',\mathfrak{P}})).$$
We enlarge $E$, if necessary so that we can fix an embedding $F'_{\mathfrak{P}} \hookrightarrow E$.
Thus we can consider the modular symbol $\phi^{*}_{\Psi}$ as a a class in $\Ho_{\mathrm{c}}(\YFss(\frkn), V_{kk}(\OO_{E}))$ well-defined up to a unit in $\OO_{E}$.

For $\delta\in\Delta^*$ and for integer $r\geq 1$, let 
\begin{equation}
    \label{poly1}
    P^{\delta}_{r,j}(T)= \left\langle \phi^{*}_{\Psi}, (U_{p})^{-r}_{*}\dfrac{1}{(a^{\sigma}-a)^{j}j!{k\choose j}^2}\sum_{t\in \Zpr} \asaieisen^{k,j}_{p^{r},\frkn,at}\otimes \delta(t)(1+T)^{\logt} \right\rangle,
\end{equation}
where $\langle, \rangle$ denotes the perfect Poincar\'e duality pairing
\[\mathrm{H}^{1}_{\mathrm{c}}(\YFss(\frkn), V_{kk}(\OO_{E})) \times \dfrac{\Ht(\YFss(\frkn),T_{kk}(\OO_{E}))}{\mathrm{Torsion}} \to \OO_{E}.\]
Note that from Theorem \ref{thmlw}, we have
\[(U_{p})^{-1}_{*} \sum_{t\in(\ZZ/p)^{\times}} \asaieisen^{k,j}_{p,\frkn,at} = (1-p^{j}(U_{p})^{-1}_{*})\asaieisen^{k,j}_{1,\frkn,a}.\]
Thus, when $\delta$ is the trivial character and $r=1$, we get
\begin{equation}
\label{trivi}
    P^{\mathrm{triv}}_{1,j}(T)= \left\langle \phi^{*}_{\Psi}, \dfrac{(1-p^{-j}(U_{p})^{-1}_{*})}{(a^{\sigma}-a)^{j} j! {k\choose j}^{2}}\asaieisen^{k,j}_{1,\frkn,a} \right\rangle.
\end{equation}

In particular, for $0\leq j\leq k$, we get
\begin{align*}
    P^{\delta}_{r,j}(T) &= \left\langle ((U_{p})^{*})^{-r}\phi^{*}_{\Psi}, \dfrac{1}{(a^{\sigma}-a)^{j}j!{k\choose j}^2}\sum_{t\in \Zpr} \asaieisen^{k,j}_{p^{r},\frkn,at}\otimes \delta(t)(1+T)^{\logt}\right\rangle,\\[1em]
    &=\left\langle a_{p}^{-r}\phi^{*}_{\Psi}, \dfrac{1}{(a^{\sigma}-a)^{j}j!{k\choose j}^2}\sum_{t\in \Zpr} \asaieisen^{k,j}_{p^{r},\frkn,at}\otimes\delta(t)(1+T)^{\logt} \right\rangle,\\[1em]
    &= \dfrac{a_{p}^{-r}}{(a^{\sigma}-a)^{j}j!{k\choose j}^2}\sum_{t\in\Zpr}\left\langle \phi^{*}_{\Psi},\asaieisen^{k,j}_{p^{r},\frkn,at} \right\rangle \delta(t)(1+T)^{\logt} \in E[T].
\end{align*}
Moreover, when $\delta$ is the trivial character and $r=1$, we get
\[P^{\mathrm{triv}}_{1,j} =\dfrac{1}{(a^{\sigma}-a)^{j}j!{k\choose j}^2}\left( 1- \dfrac{p^j}{a_{p}}\right)\langle \phi^{*}_{\Psi}, \asaieisen^{k,j}_{1,\frkn,a}\rangle. \]
Let us denote $v_{p}(a_{p})$ by $n$.
\begin{lem}
\label{polylem}
For any integer $r\geq 1$, any integer $0\leq j \leq k$, and any character $\delta\in\Delta^*$, we have
\begin{enumerate}
    \item $\sup||p^{nr}P^{\delta}_{r,j}(T)||<\infty$, 
    \label{1}\vspace{0.9mm}
    \item $P^{\delta}_{r+1,j}(T) \equiv P^{\delta}_{r,j} \mod \omega_{r-1}(T),$\label{2}\vspace{0.9mm}
    \item $\sup_{r}\Big{|}\Big{|} p^{(n-j)r} \sum_{i=0}^{j}(-1)^{i} {j\choose i} P^{\delta}_{r,i}(u^{-i}(1+T)-1)\Big{|}\Big{|}<\infty$.\label{3}
\end{enumerate}
\end{lem}
\begin{proof}
    Statement \eqref{1} follows from the definition of the polynomial $P^{\delta}_{r,j}(T)$.

    Note that from the remark \ref{remnormcomp}, after converting it in the polynomial setup, we deduce
    \begin{align}
    \begin{split}
        &\sum_{t\in(\ZZ/p^{r+1})^\times}\asaieisen^{k,j}_{p^{r+1},\frkn,at}\otimes \delta(t)(1+T)^{\logt} \\ &\equiv \sum_{t\in\Zpr} \asaieisen^{k,j}_{p^{r},\frkn,at}\otimes \delta(t)(1+T)^{\logt} \mod \omega_{r-1}(T)(\Ht(\YFss(\frkn), T_{kk}(\OO_{E}))\otimes \OO_{E}[T])\label{eqq}
        \end{split}
    \end{align}
Thus, after pairing both sides of \eqref{eqq} with the modular symbol $\phi^{*}_{\Psi}$, we get the second statement.

For the third part, for simplicity assume $\delta$ is trivial and integer $r>1$. We can simplify the expression
\begin{align*}
    &\sum_{i=0}^{j}(-1)^{i}{j\choose i}P^{\delta}_{r,i}(u^{-i}(1+T)-1)\\[1em]
    &=\sum_{i=0}^{j}(-1)^{i}{j\choose i} \dfrac{a_{p}^{-r}}{(a^{\sigma}-a)^{i}i!{k\choose i}^2}\sum_{t\in\Zpr}\left\langle \phi^{*}_{\Psi},\asaieisen^{k,i}_{p^{r},\frkn,at} \right\rangle ((u)^{-i}(1+T))^{\logt},\\[1em]
    &=\sum_{i=0}^{j}(-1)^{i}{j\choose i} \dfrac{a_{p}^{-r}}{(a^{\sigma}-a)^{i}i!{k\choose i}^2}\sum_{t\in\Zpr}\left\langle \phi^{*}_{\Psi},\asaieisen^{k,i}_{p^{r},\frkn,at} \right\rangle (t)^{-i}(1+T)^{\logt},\\[1em]
    &= \left\langle \phi^{*}_{\Psi},\hspace{1mm}\sum_{i=0}^{j}(-1)^{i}{j\choose i} \dfrac{a_{p}^{-r}}{(a^{\sigma}-a)^{i}i!{k\choose i}^2}\sum_{t\in\Zpr}\asaieisen^{k,i}_{p^{r},\frkn,at}\otimes t^{-i}(1+T)^{\logt} \right\rangle 
\end{align*}
The statement then follows from Lemma \ref{patch3} and $n=v_{p}(a_{p})$. For the non-trivial character $\delta$, the proof is the same.    
\end{proof}

Hence from Lemma \ref{patchlemm} and Lemma \ref{polylem}, we deduce
\begin{thm}
\label{thmasai1}
    For any character $\delta\in\Delta^{*}$, there exists a unique polynomial sequence $P^{\delta}_{r}(T)\in E[T]$ of degree $<(k+1)p^{r-1}$ such that
    \begin{enumerate}
        \item $P^{\delta}_{r}(T) \equiv P^{\delta}_{r,j}(u^{-j}(1+T)-1) \mod \omega_{r-1}(u^{-j}(1+T)-1)$,
        \item $P^{\delta}_{r+1}(T)\equiv P^{\delta}_{r}(T) \mod \prod_{i=0}^{k}\omega_{r-1}(u^{-i}(1+T)-1)$,
        \item $\sup_{r}||p^{nr}P^{\delta}_{r}||<\infty$.
    \end{enumerate}
  Moreover, the sequence $P^{\delta}_r$ converges to $\prescript{}{c}{L^{\mathrm{As},\delta}_{p}}(\Psi)\coloneqq \lim_{r\to\infty} P^{\delta}_{r}\in\HH_{E,n}$ and
  \[\asaipaddel(\Psi)\equiv P^{\delta}_{r,j}(u^{-j}(1+T)-1)\mod\omega_{r-1}(u^{-j}(1+T)-1).\] 
\end{thm}

Note that, for any real number $w\geq 0$, $\HH_{E,w}(\Zp^\times) \cong \HH_{E,w}(\GG)$ via identification $T\mapsto \gamma_{0}-1$.

\begin{defi}[$p$-adic Asai distribution]
    For a $p$-non-ordinary small slope Bianchi eigenform $\Psi$ of level $\frkn$, weight $(k,k)$, and $U_p$-eigenvalue $a_{p}$, define the $p$-adic distribution attached to $\Psi$ to be
    \[\asaipad(\Psi) = \bigoplus_{\delta\in\Delta^{*}} \asaipaddel(\Psi) \in \HH_{E,n}(\GG),\]
    where $\asaipaddel(\Psi)$ (after identifying $T$ with $\gamma_{0}-1$) are from Theorem \ref{thmasai1} and $n=v_{p}(a_p)$ such that $0<n<k+1$.
\end{defi}

\begin{Not}
    For a Dirichlet character $\theta$ of conductor $p^{r}>1$ and any integer $0 \leq j\leq k$, write
    \begin{align*}
        \asaipad(\Psi,\theta, j) &\coloneqq \asaipad(\Psi)(\chi^{j}\theta), 
    \end{align*}
    i.e. evaluating $\asaipad(\Psi)$ at $\gamma_{0}=u^{j}\cdot\tilde{\zeta}\cdot\zeta_{p^{r-1}}$, where $\tilde{\zeta}$ is a $(p-1)$-th root of unity, corresponding to $\delta\in\Delta^{*}$ and $\zeta_{p^{r-1}}$is a primitive $p^{r-1}$-th root of unity.
\end{Not}

\subsection{The interpolation property}
We will prove that the $p$-adic distribution $\asaipad(\Psi)$ interpolates the critical $L$-values of the Asai $L$-function attached to $\Psi$.
We use Section $7C$ and Theorem 7.5 of \cite{CW} for the interpolation.

\begin{thm}
    \label{interpolation}
    Let $\Psi$ be a $p$-non-ordinary small slope Bianchi cusp form of  weight $(k,k)$, nebentypus $\epsilon_{\Psi}$, and level $\frkn$, where all primes above $p$ divides $\frkn$. 
    Let $a_{p}=c(p\OO_{F},\Psi)$ be the $U_{p}$-eigenvalue of $\Psi$ which satisfies $0< v_{p}(a_{p})< k+1$. 
    For any Dirichlet character $\theta$ of conductor $p^{r}$ and any integer $0\leq j \leq k$, $\asaipad(\Psi)$ satisfies the following interpolation property:
    \begin{enumerate}
        \item If $(-1)^j\theta(-1)=1$, then
        \[\asaipad(\Psi,\theta,j) = \dfrac{C(c,k,j)G(\theta)}{\Omega_{\Psi}}\times \mathfrak{e}_{p}(\Psi,\theta,j) L^{\mathrm{As}}(\Psi, \overline{\theta}, j+1),\]
        where
        \begin{itemize}
        \item $\mathfrak{e}_{p}(\Psi,\theta,j) =\begin{cases}
            \left(1- \dfrac{p^j}{a_{p}}\right) &\text{ if } r=0,\\[1.1em]
            \dfrac{1}{a_{p}^r} & \text{ if } r>0.
        \end{cases}$\\[1em]
            
            \item $C(c,k,j)\coloneqq \dfrac{(-1)^{k+1}\cdot p^{jr} j! (\sqrt{-D})}{2\cdot (2\pi i)^{j+1}}\cdot (c^{2}-c^{2j-2k}\epsilon_{\Psi}(c^{-1})\theta(c)^{2}),$
        where $G(\theta)$ is the Gauss sum associated with $\theta$.
        \end{itemize}
        \item If $(-1)^{j}\theta(-1)=-1$, then
        \[\asaipad(\Psi,\theta, j) = 0.\]
    \end{enumerate}
\end{thm}

\begin{proof}
    First, we assume $r>1$ and $\delta \in \Delta^{*}$. 
    Since $\Zpr\cong\Delta\times \ZZ/p^{r-1}$, write $\theta=\delta\cdot\Theta$, where $\delta$ is a character on $\Delta$ and $\Theta$ is a character on $\ZZ/p^{r-1}$.

    For any $p^{r-1}$-th root $\zeta$, $u^{j}\zeta$ is a root of polynomial $\omega_{r-1}(u^{-j}\gamma_{0}-1)$.

    Now for $\delta\in\Delta^{*}$, from Theorem \ref{thmasai1}, we have
    \begin{equation}
    \label{thmeqn}
        \asaipaddel(\Psi)\equiv P^{\delta}_{r,j}(u^{-j}\gamma_{0}-1) \mod \omega_{r-1}(u^{-j}\gamma_{0}-1)
    \end{equation}
    Hence, to calculate $\asaipad(\Psi,\theta,j)$ is equivalent to evaluating $\asaipaddel(\Psi)$ at $u^{j}\Theta$, ie, by putting $\gamma_{0}=u^{j}\zeta$, for $p^{r-1}$-th root of unity $\zeta$.

    From equation \eqref{thmeqn}, it sufficient to calculate $P^{\delta}_{r,j}(\zeta-1)$.
    Therefore, we get
    \begin{align*}
        P^{\delta}_{r,j}(\zeta-1)&= \dfrac{a_{p}^{-r}}{(a^{\sigma}-a)^{j}j!{k\choose j}^2}\sum_{t\in\Zpr}\left\langle \phi^{*}_{\Psi},\asaieisen^{k,j}_{p^{r},\frkn,at} \right\rangle \delta(t)(\zeta)^{\logt}, \\[1em]
        &=\dfrac{a_{p}^{-r}}{(a^{\sigma}-a)^{j}j!{k\choose j}^2}\sum_{t\in\Zpr}\left\langle \phi^{*}_{\Psi},\asaieisen^{k,j}_{p^{r},\frkn,at} \right\rangle \delta(t)\Theta(t),\\[1em]
        &= \dfrac{a_{p}^{-r}}{(a^{\sigma}-a)^{j}j!{k\choose j}^2}\sum_{t\in\Zpr}\left\langle \phi^{*}_{\Psi},\asaieisen^{k,j}_{p^{r},\frkn,at} \right\rangle \theta(t),
    \end{align*}
    since $\logt$ is a unique non-negative integer $< p^{r-1}$ and $\Theta$ is a character on $\ZZ/p^{r}$.
    Hence, we have deduced
    \[\asaipad(\Psi,\theta,j) = \dfrac{a_{p}^{-r}}{(a^{\sigma}-a)^{j}j!{k\choose j}^2}\sum_{t\in\Zpr}\left\langle \phi^{*}_{\Psi},\asaieisen^{k,j}_{p^{r},\frkn,at} \right\rangle \theta(t). \]

    Now, if $r=1$ and $\delta$ is non-trivial, following the calculations similar to the above case, we get
    \[\asaipad(\Psi, \theta, j) = \dfrac{a_{p}^{-1}}{(a^{\sigma}-a)^{j} j! {k\choose j}^{2}} \sum_{t\in (\ZZ/p)^{\times}} \langle \phi^{*}_{\Psi}, \asaieisen^{k,j}_{p,\frkn, at} \rangle \delta(t).\]

    Lastly, if $r=1$ and $\delta$ is the trivial character, then $\asaipad(\Psi,j)$ is evaluating $\asaipad(\Psi)$ at $u^{j}$.
    Therefore, we get
    \begin{align*}
        \asaipad(\Psi,j) &= P^{\mathrm{triv}}_{1,j}(u^{j}-1), \\
        &= \dfrac{1}{(a^{\sigma}-a)^{j}j!{k\choose j}^2}\left( 1- \dfrac{p^j}{a_{p}}\right)\langle \phi^{*}_{\Psi}, \asaieisen^{k,j}_{1,\frkn,a}\rangle.
    \end{align*}
    The theorem then follows from \cite[Theorem 7.5]{CW}.
\end{proof}

\subsection{$p$-adic distributions without $c$}
\label{withoutc}
Recall Kato's Siegel unit $\prescript{}{c}{g_{N}}\in\OO(Y_{\QQ,1}(N))^\times$ from Section \ref{asaieisendef}.
Note that if $c,d\in\ZZ_{\geq1}$ coprime to $6N$, then we have
\[(d^{2} - \langle d \rangle)\prescript{}{c}{g_{N}} = (c^{2} - \langle c \rangle)\prescript{}{d}{g_{N}}. \]
Thus, the dependence on $c$ can be removed after extending scalars to $\QQ$.
More precisely, there is an element $g_{N}\in\OO(Y_{\QQ,1}(N))^{\times}\otimes\QQ$ such that $\prescript{}{c}{g_{N}}=(c^{2}-\langle c \rangle)\cdot g_{N}$ for any choice of $c$.
Similarly, for any integer $k>0$, we have
\[\ceis^{k}_{N}=(c^{2}- c^{-k}\langle c \rangle)\mathrm{Eis}^{k}_N\in \Ho(Y_{\QQ,1}(N), T_{k}(\Qp)),\]
where $\ceis^{k}_{N}\in\Ho(Y_{\QQ,1}(N), T_{k}(\Zp))$ appearing in the Definition \ref{twstdasai}.
See \cite[Theorem 4.4.4]{KLZ1} for the details.
Following this path, we prove:
\begin{prop}
    \label{crid}
    Let $\epsilon_{\Psi}: (\OO_{F}/\frkn)^\times \to \OO_{E}^\times$ be the nebentypus of the Bianchi cusp form $\Psi$.
    Assume:
    \begin{enumerate}
        \item the restriction $\epsilon_{\Psi}|_{(\ZZ/N)^\times}$ is non-trivial;
        \item $\epsilon_{\Psi}|_{(\ZZ/N)^\times}$ does not have $p$-power conductor. 
    \end{enumerate}
    Then there exists a distribution $L^{\mathrm{As}}_{p}(\Psi)\in \HH_{E,v_{p}(a_{p})}(\GG)$. such that
    $$L^{\mathrm{As}}_{p}(\Psi) = \dfrac{1}{\left(c^{2}-c^{-2k}\epsilon_{\Psi}(c^{-1})\left(\gamma_{0}^{\log_{u}(c)}\right)^{2}\right)}\cdot\asaipad(\Psi)$$
for all valid $c\in\ZZ_{>1}$, where $\log_{u}(c)$ is the unique positive integer such that $(u\tilde{zeta}^{\log_{u}(c)}=c$ for appropriate $(p-1)$-th root of unity.
\end{prop}
Before proving this proposition, note that $L^{\mathrm{As}}_{p}(\Psi)$ satisfies the interpolation property described in Theorem \ref{interpolation} without the factor involving $c$. 
In particular, for any integer $0\leq j \leq k$ and for any Dirichlet character $\theta$ of conductor $p^{r}$, we have
\begin{equation}
\label{eqqq}
    L^{\mathrm{As}}_{p}(\Psi)(u^{j}\theta) = \dfrac{C'(k,j)G(\theta)}{\Omega_{\Psi}}\times\mathfrak{e}_{p}(\Psi,\theta,j) \Las(\Psi,\overline{\theta},j+1),
\end{equation}
where 
\[C'(k,j)=\begin{cases} \dfrac{(-1)^{k+1}\cdot p^{jr} j! (\sqrt{-D})}{2\cdot (2\pi i)^{j+1}} &\text{ if } (-1)^{j}\theta(-1)=1, \\[1em]
0 &\text{ if } (-1)^{j}\theta(-1)=-1, \end{cases}\]
and $\mathfrak{e}_{p}(\Psi,\theta,j)$ is the same one appearing in Theorem \ref{interpolation}. 

\begin{proof}[Proof of Proposition \ref{crid}]
   We follow \cite[Proposition 6.7]{CW}.
   Denote $\left(c^{2}-c^{-2k}\epsilon_{\Psi}(c^{-1})\left(\gamma_{0}^{\log_{u}(c)}\right)^{2}\right)$ by $\upsilon_c$.
   For any positive integers $c,d$ coprime to $6Np$, consider 
   \[\upsilon_d\cdot\asaipad(\Psi).\]
   Then $$\left((\upsilon_d\cdot\asaipad(\Psi))-(\upsilon_c\cdot\prescript{}{d}{L_{p}^{\mathrm{As}}(\Psi)})\right)(\chi^{j}\theta)=0$$
   for any integer $0\leq j \leq k$ and for any Dirichlet character $\theta$ of conductor $p^r$.
   Thus, by the uniqueness of the construction $\asaipad(\Psi)$, we can conclude
   \[(\upsilon_d\cdot\asaipad(\Psi)) = (\upsilon_c\cdot\prescript{}{d}{L_{p}^{\mathrm{As}}(\Psi)}).\]
   In other words, $\upsilon_{d}\cdot\asaipad(\Psi)$ is symmetric in $c$ and $d$.
   
   Now we can choose $d$ such that $\upsilon_d$ is a unit in $E\otimes\OO_{E}[[\GG]]$, since the conductor of $\epsilon_{\Psi}$ is not a power of $p$. 
   Note that we get $\epsilon_{\Psi}(d^{-1})$, since the dual of the diamond operator $\langle d \rangle$ is $\langle d^{-1} \rangle$ under the perfect Poincar\'e duality.
   
   Therefore, if we define
   \begin{equation}
       L_{p}^{\mathrm{As}}(\Psi)\coloneqq \dfrac{1}{\left(c^{2}-c^{-2k}\epsilon_{\Psi}(c^{-1})\left(\gamma_{0}^{\log_{u}(c)}\right)^{2}\right)}\asaipad(\Psi),
   \end{equation}
   then $L^{\mathrm{As}}_{p}(\Psi)$ is independent of $c$ and lies in $\HH_{E,v_{p}(a_{p})}(\GG)$.
\end{proof}

\begin{rem}
    There is a typo in the equation of $L_{p}^{\mathrm{As}}(\Psi)$ appearing in \cite[Proposition 6.7]{CW}. 
    There should be $\epsilon_{\Psi}(c^{-1})$ instead of $\epsilon_{\Psi}(c)$.
\end{rem}

\begin{rem}
    Even if $\epsilon_{\Psi}|_{(\ZZ/N)^{\times}}$ has $p$-power conductor, we can still define $$L^{\mathrm{As}}_{p}(\Psi)=\dfrac{1}{\left(c^{2}-c^{-2k}\epsilon_{\Psi}(c^{-1})\left(\gamma_{0}^{\log_{u}(c)}\right)^{2}\right)}\asaipad(\Psi),$$ but this element will not be in the distribution module $\HH_{E, v_{p}(a_{p})}(\GG)$ but in the fraction field of $\HH_{E,v_{p}(a_{p})}(\GG)$.
    This new "distribution" will have poles, i.e., it will be a meromorphic function. 
\end{rem}

\section{Signed $p$-adic Asai $L$-functions of Bianchi modular forms}
\label{dec}
This section addresses the factorization of unbounded $p$-adic $L$-functions (i.e. $p$-adic distributions) into bounded signed $p$-adic $L$-functions (i.e. $p$-adic measures) in the spirit of Pollack, Sprung, and Lei--Loeffler--Zerbes.
We will apply the machinery of logarithmic matrices, developed by the author in \cite{MD}, to obtain this factorization.

\subsection{Wach modules and logarithmic matrix}
\label{logmat}
Define, for any real number $w\geq 0$, 
\begin{align*}
  \HH_{E,w}(\GG_{1})&=\left\{\sum_{n=0}^{\infty}c_{n}(\gamma_{0}-1)^{n} : \sup_{n}\dfrac{|c_{n}|_{p}}{n^{w}}<\infty\right\},
  \end{align*}
  and recall
  $$\HH_{E,w}(\GG)=\left\{\sum_{\sigma\in\Delta}\sum_{n=0}^{\infty}c_{n,\sigma}\cdot\sigma\cdot(\gamma_{0}-1)^{n} : \sup_{n}\dfrac{|c_{n,\sigma}|_{p}}{n^{w}}<\infty, \forall \sigma\in\Delta\right\}.$$
Note that
\begin{align}
    \HH_{E,w}&\cong\HH_{E,w}(\GG_{1}),\label{eq1}\\
    \HH_{E,w}(\Zp^{\times})&\cong \HH_{E,w}(\GG) \label{eq2},
\end{align}
where $\HH_{E,w}$ and $\HH_{E,w}(\Zp^{\times})$ are defined in Section \ref{cont}.
Also, let $\HH_{E}(\GG_{1})=\bigcup_{w\geq0}\HH_{E,w}(\GG_{1})$ and $\HH_{E}(\GG)=\bigcup_{w\geq0}\HH_{E,w}(\GG)$. 
Similar to $\HH_{E,w}$, if $f\in\HH_{E,w}(\GG_1)$, then we say $f$ has growth rate $O(\log_{p}^{w})$.

Let $X$ be a variable and write $\Brig$ for the ring of power series $f(X)\in \Qp[[X]]$ such that $f$ converges everywhere inside the unit $p$-adic disk. We equip $\Brig$ with the actions by the Frobeneius $\varphi: X \mapsto (1+X)^{p}-1$ and $\sigma: X\mapsto (1+X)^{\chi(\sigma)}-1$ for $\sigma\in\GG$, where $\chi$ is the $p$-adic cyclotomic character such that $\chi(\gamma_{0})=u$.
Let $\BrigE=E\otimes\Brig$ and inside $\BrigE$ there is a subring $\AAA^{+}_{E}=\OO_{E}[[X]]$ which is also equipped with the actions of $\varphi$ and $\GG$.  
Note that there exists a $\Lambda_{E}(\GG)$-module isomorphism between $(\BrigE)^{\psi=0}$ and $\HH_{E}(\GG)$ called the Mellin transform, where $\psi$ is a left inverse of $\varphi$ such that $\varphi\circ\psi(f(X))=\frac{1}{p}\sum_{\zeta^{p}=1}f(\zeta(1+X)-1)$.
Moreover, we can identify $(\AAA^{+}_{E})^{\psi=0}$ with $\Lambda_{\OO_{E}}(\GG)$. 
See \cite[B.2.8]{PR1} for more details.

Fix $a\in\OO_{E}$ with $v_{p}(a) > \left\lfloor \dfrac{k}{p-1} \right\rfloor$ and $k\geq 0$ be an integer.
Let $\alpha,\beta$ be the distinct roots of polynomial $X^{2}-aX+vp^{k+1}$, for some $v$ such that $v^{1/2}\in\OO_{E}^\times$. 
Then by the methods in \cite{MD} (which are based on methods in \cite{BLZ}), there exists a $E$-linear crystalline $\Gal(\overline{\Qp}/\Qp)$-representation $V$ and an $\OO_{E}$-stable lattice $T$ in $V$ such that:
\begin{enumerate}
    \item there exists an $\OO_{E}$-basis $v_{1}, v_{2}$ of Dieudonn\'e module $\Dcris(T)$ such that the matrix of $\varphi$ with respect to this basis is
    \[A_{\varphi} =\MatA.\]
    \item we have a Wach module $\NN(T)$ with an $\AAA^{+}_{E}$-basis $n_{1}, n_{2}$ such that, for $i=1,2$, 
    \[n_{i}\equiv v_{i} \mod X.\]
    See \cite{BER1}, \cite{BLZ}, \cite{LLZ2}, and \cite{MD} for definitions and more details about Wach modules.
    \item  $(\vp^{*}\NN(T))^{\psi=0}$ is a $\Lambda_{\OO_{E}}$-module with basis $(1+\pi)\vp(n_{1}), (1+\pi)\vp(n_{1})$, where $(\vp^{*}\NN(T))^{\psi=0}$ is $\AAA^{+}_{E}$-submodule of $\NN(T)[X^{-1}]$ generated by $\vp(\NN(T))$
    \item $(\BrigE)^{\psi=0}\otimes\Dcris(T)$ is an $\HH_{E}(\GG)$-module with basis $(1+\pi)\otimes v_{1}, (1+\pi)\otimes v_{2}$.
\end{enumerate}

\begin{defi}[Logarithmic matrix]
We define the $2\times2$ logarithmic matrix $\Mbar\in M_{2,2}(\HH_{E}(\GG))$ as the change of the basis matrix for the following $\HH_{E}(\GG)$-module homomorphism:
\[(\vp^{*}\NN(T))^{\psi=0} \hookrightarrow (\BrigE)^{\psi=0}\otimes\Dcris(T),\]
i.e.,
\begin{equation}
    \begin{bmatrix}
        (1+\pi)\vp(n_{1}) & (1+\pi)\vp(n_{1})  
    \end{bmatrix}=\begin{bmatrix}
        (1+\pi)\otimes v_{1} & (1+\pi)\otimes v_{2}  
    \end{bmatrix}\Mbar.
\end{equation}
\end{defi}
Note that $\Mbar$ is unique and in fact $\Mbar\in M_{2,2}(\HH_{E}(\GG_1))$. 
See \cite[Section 4]{MD} and \cite{LLZ2} for the details about the construction of $\Mbar$.
\begin{prop}
\label{mbarprop}
    \hfill
    \begin{enumerate}
        \item The elements in the first row of matrix $Q^{-1}\Mbar$ have growth $O(\log^{v_{p}(\alpha)}_{p})$ and the elements in the second row have growth $O(\log_{p}^{v_{p}(\beta)})$, where $Q=\MatQ$.\vspace{0.9mm}
        \item The second row of $A_{\vp}^{-n}\Mbar$ is divisible by $\prod_{i=0}^{k} \Phi_{n-1}(u^{-i}\gamma_{0}-1)$ over $\HH_{E}(\GG_1)$, where $\Phi_{n-1}(T)=\dfrac{\omega_{n-1}(T)}{\omega_{n-2}(T)}$ is the $p^{n-1}$-th cyclotomic polynomial.\vspace{0.9mm}
        \item The determinant of $\Mbar$ is $\dfrac{\log_{p,k+1}(\gamma_{0})}{\delta_{k+1}(\gamma_{0})}$ up to a unit in $\Lambda_{E}(\GG_{1})$.
        Here $\log_{p,k+1}(\gamma_{0})=\prod_{i=0}^{k}\log_{p}(u^{-i}\gamma_{0})$ and $\delta_{k+1}(\gamma_{0})= \prod_{i=0}^{k}(u^{-i}\gamma_{0}-1)$.\vspace{0.9mm}
        \item Moreover, $\det(\Mbar)$ is $O(\log_{p}^{k+1})$ and $\log^{k+1}_{p}$ is $O(\det(\Mbar))$. 
    \end{enumerate}
\end{prop}
\begin{proof}
    See \cite[Proposition 5.2, Lemma 5.2, Lemma5.3, Lemma 5.4]{MD} for the details.
    Note that in \cite{MD}, $k\geq 2$. Here, we are taking $k\geq0$, but otherwise, everything is identical.
\end{proof}


\subsection{Decomposition of the distribution into measures}
\label{decomp}
Assume prime $p$ splits in $F$ as $p\OO_{F}=\PP\PPP$. We enlarge $E/\Qp$, if necessary, so that it contains $F$ and all Hecke eigenvalues of $\Psi$.

Let $\Psi$ be a Bianchi cusp form of weight $(k,k)$ and level $\mathcal{N}$ coprime to $p$. 

Furthermore, we assume that:
\begin{enumerate}
    \item $\Psi$ is $\PP$-non-ordinary and $\PPP$-ordinary. In other words, for $\qqq\in\pset$, if $a_{\qqq}$ is the $T_{\qqq}$ Hecke eigenvalue, then $v_{p}(a_{\PP})>0$ and $v_{p}(a_{\PPP})=0$.
    \item $v_{p}(a_{\PP})> \left\lfloor \dfrac{k}{p-1} \right\rfloor$.
\end{enumerate}

For the simplicity of calculations, assume the nebentypus of $\Psi$ is trivial.
For $\qqq\in\pset$, let $\alpha_{\qqq}$ and $\beta_{\qqq}$ be the roots of Hecke polynomial $X^{2}-a_{\qqq}X + p^{k+1}$.
We also assume $\alpha_{\PP}\neq\beta_{\PP}$.
Since we have assumed $\Psi$ is $\PPP$-non-ordinary, we know $\alpha_{\PPP}$  a $p$-adic unit and $v_{p}(\beta_{\PPP})=k+1$.
We also know that $v_{p}(\alpha_{\PP}), v_{p}(\beta_{\PP})<k+1$. 
Let $\tilde{\alpha}=\alpha_{\PPP}\alpha_{\PP}$ and $\tilde{\beta}=\alpha_{\PPP}\beta_{\PP}$.
Then $v_{p}(\rta)=v_{p}(\alpha_{\PP})$ and $v_{p}(\rtb)=v_{p}(\beta_{\PP})$.

We consider the following two $p$-stabilizations of $\Psi$:
\[\Psi^{\rta}=\Psi^{\alpha_{\PPP}\alpha_{\PP}} \text{ and } \Psi^{\rtb}=\Psi^{\alpha_{\PPP}\beta_{\PP}},\]
such that $\Psi^{\rta}$ and $\Psi^{\rtb}$ are of level $p\mathcal{N}$ and
\begin{align*}
    U_{p}(\Psi^{\rta})&=\rta\cdot\Psi^{\rta},\\
    U_{p}(\Psi^{\rtb})&=\rtb\cdot\Psi^{\rtb}.
\end{align*}
For more details about the $p$-stabilizations of Bianchi modular forms, we refer to \cite[Section 3.3]{LPS}.

Thus, for $\dagger\in\{\rta, \rtb\},$ from Theorem \ref{thmasai1} and Theorem \ref{interpolation}, we can attach a $p$-adic distribution $\asaipad(\Psi^{\dagger})$ to $\Psi^{\dagger}$ such that
\begin{itemize}
    \item \textit{(Growth property)}$\asaipad(\Psi^{\dagger}) \in \HH_{E, v_{p}(\dagger)}(\GG)$;\vspace{0.9mm}
    \item \textit{(Interpolation)} for any Dirichlet character $\theta$ of conductor $p^r$ and any integer $0\leq j \leq k$, we have
    \[\asaipad(\Psi^{\dagger},\theta, j) = \dfrac{*}{\dagger^{r}}\Las(\Psi^{\dagger}, \overline{\theta}, j+1),\]
    where 
    \[*=\begin{cases}
        \text{some non-zero explicit constant independent of }\dagger &\text{ if } (-1)^{j}\theta(-1)=1;\\
        0 &\text{ if } (-1)^{j}\theta(-1)=-1.
    \end{cases}\]
\end{itemize}

We have the following decomposition theorem:
\begin{thm}
  \label{thmasai2}
  There exist $\prescript{}{c}{L^{\mathrm{As},\sharp}_{p}}, \prescript{}{c}{L^{\mathrm{As},\flat}_{p}} \in E\otimes\OO_{E}[[\GG]]$ and a logarithmic matrix $\tilde{\Mbar}\in M_{2,2}(\HH_{E}(\GG))$ such that 
  \begin{equation}
      \begin{pmatrix}
          \asaipad(\Psi^{\rta})\\[1em]
          \asaipad(\Psi^{\rtb})
      \end{pmatrix} = \tilde{Q}^{-1}\tilde{\Mbar}\begin{pmatrix}
          \prescript{}{c}{L^{\mathrm{As},\sharp}_{p}}\\[1em]
          \prescript{}{c}{L^{\mathrm{As},\flat}_{p}}
      \end{pmatrix},
  \end{equation}\\\\
  where $\tilde{Q}=\begin{pmatrix}
      \rta & -\rtb\\[0.5em]
      -\alpha^{2}_{\PPP}p^{k+1} & \alpha^{2}_{\PPP}p^{k+1}
  \end{pmatrix}$, and $\tilde{\Mbar}$ satisfies properties from Proposition \ref{mbarprop} after replacing $\alpha,\beta$ with $\rta,\rtb$ respectively. 
\end{thm}

\begin{proof}
    Since 
    \begin{align*}
        v_{p}(\rta+\rtb)&=v_{p}(\alpha_{\PPP}(\alpha_{\PP}+\beta_{\PP})),\\
        &= v_{p}(\alpha_{\PPP})+v_{p}(a_{\PP}),\\
        &= v_{p}(a_{\PPP})>\left\lfloor \dfrac{k}{p-1} \right\rfloor,
    \end{align*}
    and for $\dagger\in\{\rta,\rtb\}$, $\asaipad(\Psi^{\dagger})$ has growth rate $O(\log_{p}^{v_{p}(\dagger)})$ and interpolation property $$\asaipad(\Psi^{\dagger},\theta,j)=\dfrac{c_{\theta,j}}{\dagger^{r}},$$
    where $c_{\theta,j}$ is independent of $\dagger$, the theorem follows from \cite[Theorem 5.5]{MD}.
\end{proof}

We conclude this subsection with some remarks.
\begin{rem}
We will explain briefly why we considered the polynomial $X^{2}- a_{\qq}X+ p^{k+1}$.
This is because of the Euler factors of the Asai $L$-function when prime $p$ splits in $F$.
Recall from the equation \eqref{lfactors}, when $p$ splits as $\PP\PPP$ in $\OO_F$ and is coprime to the level $\mathcal{N}$, the local $L$-factor at $p$ of the Asai $L$-function $L^{\mathrm{As}}(\Psi,s)$ is 
\[\dfrac{1}{P^{\mathrm{As}}_{p}(\Psi,s)}= (1-\alpha_{\PP}\alpha_{\PPP}p^{-s})(1-\alpha_{\PP}\beta_{\PPP}p^{-s})(1-\beta_{\PP}\alpha_{\PPP}p^{-s})(1-\beta_{\PP}\beta_{\PPP}p^{-s}),\]
where $\alpha_{\qqq},\beta_{\qqq}$ are the roots of polynomial $X^{2}-a_{\qqq}X + p^{k+1}$, for $\qqq\in\pset$.
For Galois representation theoretic interpretation, see \cite[Section 4]{Ghate1}.
\end{rem}

\begin{rem}
    We can assume $\Psi$ to be $\PP$-ordinary and $\PPP$-non-ordinary. 
    If we assume that $\Psi$ is non-ordinary at both the primes $\PP$ and $\PPP$, then we might get into trouble.
    Note that we are in the finite slope situation, i.e., for $\qq\in\pset, v_{p}(a_{\qq})< k+1$.
    Now if $\alpha_{\qq},\beta_{\qq}$ are the roots of $X^{2}-a_{\qq}X+p^{k+1}$, then we have four $p$-stabilizations $\Psi^{\bullet,\dagger}$, where $\bullet\in\{\alpha_{\PP}, \beta_{\PP}\}$ and $\dagger\in\{\alpha_{\PPP}, \beta_{\PPP}\}$ and all $\Psi^{\bullet,\dagger}$ are $p$-non-ordinary. 
     For simplicity, assume $k=0$.
     Now suppose we want to decompose the distributions $\asaipad(\Psi^{\alpha_{\PP}\alpha_{\PPP}})$ and $\asaipad(\Psi^{\alpha_{\PP}\beta_{\PPP}})$ into the linear combination of bounded measures using logarithmic matrices.
    We know that $0< v_{p}(\alpha_{\qq}), v_{p}(\beta_{\qq})< 1$.  
    To use logarithmic matrices method, $v_{p}(\alpha_{\PP}^{2}\alpha_{\PPP}\beta_{\PPP})$ should be a positive integer, since from the construction and properties of the logarithmic matrix, we know $\det(\Mbar) \sim O(\log_{p}^{m})$ where $m=v_{p}(\alpha_{\PP}^{2}\alpha_{\PPP}\beta_{\PPP})\in \ZZ_{\geq 1}$.
    But then it might happen that, for example, $v_{p}(\alpha_{\PP}^{2}\alpha_{\PPP}\beta_{\PPP})$ is not an integer, since $v_{p}(\alpha_{\PP}^{2})$ maybe an element in $\QQ$ which is not an integer.

    
\end{rem}
\subsection{Signed $p$-adic Asai $L$-function without $c$}
Recall $\Psi$ is a Bianchi eigenform of weight $(k,k)$ and level $\mathcal{N}$ which is coprime to $p$, such that $\Psi$ is non-ordinary at $\PP$ and ordinary at $\PPP$. 
We furthermore assume:
\begin{enumerate}
    \item The nebentypus $\epsilon_{\Psi}:(\OO_{F}/\mathcal{N})^\times \to \overline{\QQ}^\times$ is non-trivial,
    \item the restriction $\epsilon_{\Psi}|_{(\ZZ/M)^\times}$ is non-trivial, where $M\in\ZZ$ such that $(M)=\ZZ\cap\mathcal{N}$,
    \item $\epsilon_{\Psi}|_{(\ZZ/M)^\times}$ does not have a $p$-power conductor.
\end{enumerate}
For $\qqq\in\pset$, let $\alpha_{\qq}$ and $\beta_{\qq}$ be the roots of Hecke polynomial
$X^{2}- a_{\qq}X + \epsilon_{\Psi}(\qq)p^{k+1},$
where $a_{\qq}$ is $T_{\qq}$-eigenvalue of $\Psi$.
We assume $\alpha_{\qq}\neq \beta_{\qq}$.
Like in the previous subsection, let $\tilde{\alpha}=\alpha_{\PPP}\alpha_{\PP}$ and $\tilde{\beta}=\alpha_{\PPP}\beta_{\PP}$ and consider two $p$-stabilizations $\Psi^{\tilde{\alpha}}$ and $\Psi^{\tilde{\beta}}$ of $\Psi$ with the $U_{p}$-eigenvalues $\tilde{\alpha}$ and $\tilde{\beta}$ respectively.
Note that the nebentypus will not change after the $p$-stabilization.

Thus, for $\bullet\in\{\tilde{\alpha}, \tilde{\beta}\}$, Proposition \ref{crid} implies there exists a $p$-adic distribution $L_{p}^{\mathrm{As}}(\Psi^{\bullet}) \in \HH_{E, v_{p}(\bullet)}(\GG)$ such that
\[L_{p}^{\mathrm{As}}(\Psi^{\bullet}) = \dfrac{1}{\left(c^{2}-c^{-2k}\epsilon_{\Psi}(c^{-1})\left(\gamma_{0}^{\log_{u}(c)}\right)^{2}\right)}\cdot\asaipad(\Psi^{\bullet}),\]
where $c$ is a suitable integer coprime to $6pM$,
since we have imposed the conditions on the nebentypus $\epsilon_{\Psi}$.
It satisfies the following interpolation property: for any Dirichlet character $\theta$ of conductor $p^{r}>1$ and any integer $0\leq j \leq k$, we have
    \[L^{\mathrm{As}}_{p}(\Psi^{\bullet},\theta, j) = \dfrac{*}{\bullet^{r}}\Las(\Psi^{\bullet}, \overline{\theta}, j+1),\]
    where 
    \[*=\begin{cases}
        \text{some non-zero explicit constant independent of }\bullet \text{ and }c &\text{ if } (-1)^{j}\theta(-1)=1;\\
        0 &\text{ if } (-1)^{j}\theta(-1)=-1.
    \end{cases}\]

If we assume $v_{p}(a_{\PP}+\beta_{\PP})> \left\lfloor \dfrac{k}{p-1}\right\rfloor$, then we have the following decomposition theorem similar to Theorem \ref{thmasai2}:
\begin{thm}
    \label{thmasai3}
There exist $L^{\mathrm{As},\sharp}_{p}, L^{\mathrm{As},\flat}_{p} \in E\otimes\OO_{E}[[\GG]]$ and a logarithmic matrix $\tilde{\Mbar}\in M_{2,2}(\HH_{E}(\GG))$ such that 
  \begin{equation}
      \begin{pmatrix}
          L_{p}^{\mathrm{As}}(\Psi^{\rta})\\[1em]
          L_{p}^{\mathrm{As}}(\Psi^{\rtb})
      \end{pmatrix} = \tilde{Q}^{-1}\tilde{\Mbar}\begin{pmatrix}
          L^{\mathrm{As},\sharp}_{p}\\[1em]
          L^{\mathrm{As},\flat}_{p}
      \end{pmatrix},
  \end{equation}\\\\
  where $\tilde{Q}=\begin{pmatrix}
      \rta & -\rtb\\[0.5em]
      -\alpha^{2}_{\PPP}\epsilon_{\Psi}(\PP)p^{k+1} & \alpha^{2}_{\PPP}\epsilon_{\Psi}(\PP)p^{k+1}
  \end{pmatrix}$, and $\tilde{\Mbar}$ satisfies properties from Proposition \ref{mbarprop} after replacing $\alpha,\beta$ with $\rta,\rtb$ respectively.    
  Moreover, for suitable $c$, we have
  \begin{equation*}
      \begin{pmatrix}
          L^{\mathrm{As},\sharp}_{p}\\[1em]
          L^{\mathrm{As},\flat}_{p}
      \end{pmatrix}=
      \dfrac{1}{\left(c^{2}-c^{-2k}\epsilon_{\Psi}(c^{-1})\left(\gamma_{0}^{\log_{u}(c)}\right)^{2}\right)}\begin{pmatrix}
          \prescript{}{c}{L^{\mathrm{As},\sharp}_{p}}\\[1em]
          \prescript{}{c}{L^{\mathrm{As},\flat}_{p}}
      \end{pmatrix}.
  \end{equation*}
\end{thm}

\begin{proof}
    Similar to the proof of Theorem \ref{thmasai2}.
\end{proof}

\section*{Acknowledgments}
I would like to thank my PhD advisor, Antonio Lei, for suggesting this topic to me and for all the helpful discussions, patience, support, and encouragement he provided throughout the completion of this article. 
I would also like to thank Chris Williams for the useful discussions about \cite{CW} and for the clarifications regarding the locally symmetric spaces associated with Bianchi modular forms.  
Finally, I would like to thank the anonymous referee for their valuable comments and suggestions, which greatly helped in improving the exposition.

\nocite{*} 
\printbibliography[title={References}]

@article {CW,
    AUTHOR = {Loeffler, David and Williams, Chris},
     TITLE = {{$p$}-adic {A}sai {$L$}-functions of {B}ianchi modular forms},
   JOURNAL = {Algebra Number Theory},
  FJOURNAL = {Algebra \& Number Theory},
    VOLUME = {14},
      YEAR = {2020},
    NUMBER = {7},
     PAGES = {1669--1710},
}

@article{Will,
   AUTHOR = {Williams, Chris},
     TITLE = {{$P$}-adic {$L$}-functions of {B}ianchi modular forms},
   JOURNAL = {Proc. Lond. Math. Soc. (3)},
  FJOURNAL = {Proceedings of the London Mathematical Society. Third Series},
    VOLUME = {114},
      YEAR = {2017},
    NUMBER = {4},
     PAGES = {614--656},
 }

@article {Ghate1,
    AUTHOR = {Ghate, Eknath},
     TITLE = {Critical values of the twisted tensor {$L$}-function in the
              imaginary quadratic case},
   JOURNAL = {Duke Math. J.},
  FJOURNAL = {Duke Mathematical Journal},
    VOLUME = {96},
      YEAR = {1999},
    NUMBER = {3},
     PAGES = {595--638},
}

@article{Kato,
     AUTHOR = {Kato, Kazuya},
     TITLE = {{$p$}-adic {H}odge theory and values of zeta functions of
              modular forms},
      NOTE = {Cohomologies $p$-adiques et applications arithm\'{e}tiques.
              III},
   JOURNAL = {Ast\'{e}risque},
  FJOURNAL = {Ast\'{e}risque},
    NUMBER = {295},
      YEAR = {2004},
     PAGES = {ix, 117--290},
}

@article {KLZ1,
    AUTHOR = {Kings, Guido and Loeffler, David and Zerbes, Sarah Livia},
     TITLE = {Rankin-{E}isenstein classes and explicit reciprocity laws},
   JOURNAL = {Camb. J. Math.},
  FJOURNAL = {Cambridge Journal of Mathematics},
    VOLUME = {5},
      YEAR = {2017},
    NUMBER = {1},
     PAGES = {1--122},
}

@article {LLZ_ann,
    AUTHOR = {Lei, Antonio and Loeffler, David and Zerbes, Sarah Livia},
     TITLE = {Euler systems for {R}ankin-{S}elberg convolutions of modular
              forms},
   JOURNAL = {Ann. of Math. (2)},
  FJOURNAL = {Annals of Mathematics. Second Series},
    VOLUME = {180},
      YEAR = {2014},
    NUMBER = {2},
     PAGES = {653--771},
}

@article{LLZ1,
AUTHOR = {Lei, Antonio and Loeffler, David and Zerbes, Sarah Livia},
     TITLE = {Wach modules and {I}wasawa theory for modular forms},
   JOURNAL = {Asian J. Math.},
  FJOURNAL = {Asian Journal of Mathematics},
    VOLUME = {14},
      YEAR = {2010},
    NUMBER = {4},
     PAGES = {475--528},
   %   ISSN = {1093-6106,1945-0036},
   %MRCLASS = {11R23 (11F80 11S31)},
  %MRNUMBER = {2774276},
%MRREVIEWER = {Gergely\ Z\'{a}br\'{a}di},
 %      DOI = {10.4310/AJM.2010.v14.n4.a2},
%       URL = {https://doi.org/10.4310/AJM.2010.v14.n4.a2},
}

@article{LLZ2,
   AUTHOR = {Lei, Antonio and Loeffler, David and Zerbes, Sarah Livia},
     TITLE = {Coleman maps and the {$p$}-adic regulator},
   JOURNAL = {Algebra Number Theory},
  FJOURNAL = {Algebra \& Number Theory},
    VOLUME = {5},
      YEAR = {2011},
    NUMBER = {8},
     PAGES = {1095--1131},
 }

@article {LZ_ran,
    AUTHOR = {Loeffler, David and Zerbes, Sarah Livia},
     TITLE = {Rankin-{E}isenstein classes in {C}oleman families},
   JOURNAL = {Res. Math. Sci.},
  FJOURNAL = {Research in the Mathematical Sciences},
    VOLUME = {3},
      YEAR = {2016},
     PAGES = {Paper No. 29, 53},
}

@article {LLZ_asai,
    AUTHOR = {Lei, Antonio and Loeffler, David and Zerbes, Sarah Livia},
     TITLE = {Euler systems for {H}ilbert modular surfaces},
   JOURNAL = {Forum Math. Sigma},
  FJOURNAL = {Forum of Mathematics. Sigma},
    VOLUME = {6},
      YEAR = {2018},
     PAGES = {Paper No. e23, 67},
}

@article {KLZ2,
    AUTHOR = {Kings, Guido and Loeffler, David and Zerbes, Sarah Livia},
     TITLE = {Rankin-{E}isenstein classes for modular forms},
   JOURNAL = {Amer. J. Math.},
  FJOURNAL = {American Journal of Mathematics},
    VOLUME = {142},
      YEAR = {2020},
    NUMBER = {1},
     PAGES = {79--138},
}

@article{BL1,
  AUTHOR = {B\"{u}y\"{u}kboduk, K\^{a}z\i m and Lei, Antonio},
     TITLE = {Iwasawa theory of elliptic modular forms over imaginary
              quadratic fields at non-ordinary primes},
   JOURNAL = {Int. Math. Res. Not. IMRN},
  FJOURNAL = {International Mathematics Research Notices. IMRN},
      YEAR = {2021},
    NUMBER = {14},
     PAGES = {10654--10730},
}

@article{Pollack,
  AUTHOR = {Pollack, Robert},
     TITLE = {On the {$p$}-adic {$L$}-function of a modular form at a
              supersingular prime},
   JOURNAL = {Duke Math. J.},
  FJOURNAL = {Duke Mathematical Journal},
    VOLUME = {118},
      YEAR = {2003},
    NUMBER = {3},
     PAGES = {523--558},
 }

@incollection {Kings1,
    AUTHOR = {Kings, Guido},
     TITLE = {Eisenstein classes, elliptic {S}oul\'e{} elements and the
              {$\ell$}-adic elliptic polylogarithm},
 BOOKTITLE = {The {B}loch-{K}ato conjecture for the {R}iemann zeta function},
    SERIES = {London Math. Soc. Lecture Note Ser.},
    VOLUME = {418},
     PAGES = {239--296},
 PUBLISHER = {Cambridge Univ. Press, Cambridge},
      YEAR = {2015},
}

@article {LSZ,
    AUTHOR = {Loeffler, David and Skinner, Christopher and Zerbes, Sarah
              Livia},
     TITLE = {Euler systems for {${\rm GSp}(4)$}},
   JOURNAL = {J. Eur. Math. Soc. (JEMS)},
  FJOURNAL = {Journal of the European Mathematical Society (JEMS)},
    VOLUME = {24},
      YEAR = {2022},
    NUMBER = {2},
     PAGES = {669--733},
}

@article{PR1,
  AUTHOR = {Perrin-Riou, Bernadette},
     TITLE = {Th\'{e}orie d'{I}wasawa des repr\'{e}sentations {$p$}-adiques
              sur un corps local},
      NOTE = {With an appendix by Jean-Marc Fontaine},
   JOURNAL = {Invent. Math.},
  FJOURNAL = {Inventiones Mathematicae},
    VOLUME = {115},
      YEAR = {1994},
    NUMBER = {1},
     PAGES = {81--161},
}

@article{MD,
AUTHOR = {Deo, Mihir},
     TITLE = {Signed {$p$}-adic {$L$}-functions of {B}ianchi modular forms},
   JOURNAL = {Accepted for publication in Research in Number Theory},
      YEAR = {2025},
      URL={https://arxiv.org/abs/2501.10581}
}

@incollection {AV,
    AUTHOR = {Amice, Yvette and V\'elu, Jacques},
     TITLE = {Distributions {$p$}-adiques associ\'ees aux s\'eries de
              {H}ecke},
 BOOKTITLE = {Journ\'ees {A}rithm\'etiques de {B}ordeaux ({C}onf., {U}niv.
              {B}ordeaux, {B}ordeaux, 1974)},
    SERIES = {Ast\'erisque},
    VOLUME = {No. 24-25},
     PAGES = {119--131},
 PUBLISHER = {Soc. Math. France, Paris},
      YEAR = {1975}
}

@article{Vishik,
 AUTHOR = {Vi\v{s}ik, M. M.},
     TITLE = {Nonarchimedean measures associated with {D}irichlet series},
   JOURNAL = {Mat. Sb. (N.S.)},
  FJOURNAL = {Matematicheski\u{\i} Sbornik. Novaya Seriya},
    VOLUME = {99(141)},
      YEAR = {1976},
    NUMBER = {2},
     PAGES = {248--260, 296},
 }

@article{BER1,
  AUTHOR = {Berger, Laurent},
     TITLE = {Limites de repr\'{e}sentations cristallines},
   JOURNAL = {Compos. Math.},
  FJOURNAL = {Compositio Mathematica},
    VOLUME = {140},
      YEAR = {2004},
    NUMBER = {6},
     PAGES = {1473--1498},
 }

@article{BLZ,
  AUTHOR = {Berger, Laurent and Li, Hanfeng and Zhu, Hui June},
     TITLE = {Construction of some families of 2-dimensional crystalline
              representations},
   JOURNAL = {Math. Ann.},
  FJOURNAL = {Mathematische Annalen},
    VOLUME = {329},
      YEAR = {2004},
    NUMBER = {2},
     PAGES = {365--377},
 }

@article{LPS,
  AUTHOR = {Palacios, Luis Santiago},
     TITLE = {Functional equation of the {$p$}-adic {$L$}-function of
              {B}ianchi modular forms},
   JOURNAL = {J. Number Theory},
  FJOURNAL = {Journal of Number Theory},
    VOLUME = {242},
      YEAR = {2023},
     PAGES = {725--753},
 }

@article{Spr1,
  AUTHOR = {Sprung, Florian E. Ito},
     TITLE = {Iwasawa theory for elliptic curves at supersingular primes: a
              pair of main conjectures},
   JOURNAL = {J. Number Theory},
  FJOURNAL = {Journal of Number Theory},
    VOLUME = {132},
      YEAR = {2012},
    NUMBER = {7},
     PAGES = {1483--1506},
 }

@incollection {MR4179366,
    AUTHOR = {Balasubramanyam, Baskar and Ghate, Eknath and Vangala,
              Ravitheja},
     TITLE = {{$p$}-adic {A}sai {$L$}-functions attached to {B}ianchi cusp
              forms},
 BOOKTITLE = {Modular forms and related topics in number theory},
    SERIES = {Springer Proc. Math. Stat.},
    VOLUME = {340},
     PAGES = {19--45},
 PUBLISHER = {Springer, Singapore},
      YEAR = {[2020] \copyright 2020},
 %     ISBN = {978-981-15-8719-1; 978-981-15-8718-4},
   MRCLASS = {11F67 (11F75)},
  MRNUMBER = {4179366},
}

@incollection {CPR,
    AUTHOR = {Coates, John and Perrin-Riou, Bernadette},
     TITLE = {On {$p$}-adic {$L$}-functions attached to motives over {${\bf
              Q}$}},
 BOOKTITLE = {Algebraic number theory},
    SERIES = {Adv. Stud. Pure Math.},
    VOLUME = {17},
     PAGES = {23--54},
 PUBLISHER = {Academic Press, Boston, MA},
      YEAR = {1989},
      ISBN = {0-12-177370-1},
   MRCLASS = {11G40 (11F33 11F67 11G09 11R23)},
  MRNUMBER = {1097608},
MRREVIEWER = {Nigel\ Boston},
       DOI = {10.2969/aspm/01710023},
}

@article {Coates,
    AUTHOR = {Coates, John},
     TITLE = {On {$p$}-adic {$L$}-functions attached to motives over {${\bf
              Q}$}. {II}},
   JOURNAL = {Bol. Soc. Brasil. Mat. (N.S.)},
  FJOURNAL = {Boletim da Sociedade Brasileira de Matem\'atica. Nova S\'erie},
    VOLUME = {20},
      YEAR = {1989},
    NUMBER = {1},
     PAGES = {101--112},
      ISSN = {0100-3569},
   MRCLASS = {11G40 (11F33 11F67 11G09 11R23)},
  MRNUMBER = {1129081},
MRREVIEWER = {Nigel\ Boston},
       DOI = {10.1007/BF02585471},
}

@article {HP,
    AUTHOR = {Harron, Robert and Pottharst, Jonathan},
     TITLE = {Iwasawa theory for symmetric powers of {CM} modular forms at
              nonordinary primes, {II}},
   JOURNAL = {J. Th\'eor. Nombres Bordeaux},
  FJOURNAL = {Journal de Th\'eorie des Nombres de Bordeaux},
    VOLUME = {28},
      YEAR = {2016},
    NUMBER = {3},
     PAGES = {655--677},
      ISSN = {1246-7405,2118-8572},
   MRCLASS = {11R23 (11F67 11F80)},
  MRNUMBER = {3610691},
MRREVIEWER = {Giovanni\ Rosso},
       DOI = {10.5802/jtnb.957},
}

@article {BM,
    AUTHOR = {Borel, A. and Moore, J. C.},
     TITLE = {Homology theory for locally compact spaces},
   JOURNAL = {Michigan Math. J.},
  FJOURNAL = {Michigan Mathematical Journal},
    VOLUME = {7},
      YEAR = {1960},
     PAGES = {137--159},
      ISSN = {0026-2285,1945-2365},
   MRCLASS = {55.30},
  MRNUMBER = {131271},
MRREVIEWER = {R.\ L.\ Wilder},
}

\end{document}